\documentclass[reqno]{amsart}
\usepackage[american]{babel}
\usepackage{amssymb,euscript}
\numberwithin{equation}{section}
\newtheorem{theorem}{Theorem}[section]
\newtheorem{lemma}[theorem]{Lemma}

\theoremstyle{definition}
\newtheorem{definition}[theorem]{Definition}
\theoremstyle{remark}
\newtheorem{remark}[theorem]{Remark}

\DeclareMathOperator{\sgn}{sign}

\begin{document}

\title[Zakharov--Kuznetsov Equation]
{Initial-Boundary Value Problems in a Rectangle for Two-Dimensional Zakharov--Kuznetsov Equation}

\author[A.V.~Faminskii]{Andrei~V.~Faminskii}

\thanks{The publication was financially supported by the Ministry of Education and Science of the Russian Federation (the Agreement 02.A03.21.0008 and the Project 1.962.2017/PCh)}

\address{RUDN University, 6 Miklukho--Maklaya Street, Moscow, 117198, Russia}

\email{afaminskii@sci.pfu.edu.ru}

\subjclass[2010]{Primary 35Q53; Secondary 35B40}

\keywords{Zakharov--Kuznetsov equation, initial-boundary value problem, global solution, decay, controllability}

\date{}

\begin{abstract}
Initial-boundary value problems in a bounded rectangle with different types of boundary conditions for two-dimensional Zakharov--Kuznetsov equation are considered. Results on global well-posedness in the classes of weak and regular solution are established. As applications of the developed technique results on boundary controllability and long-time decay of weak  solutions are also obtained.
\end{abstract}

\maketitle

\section{Introduction. Description of main results}\label{S1}
The two dimensional Zakharov--Kuznetsov equation (ZK)
\begin{equation}\label{1.1}
u_t+bu_x+u_{xxx}+u_{xyy}+uu_x=f(t,x,y)
\end{equation}
($b$ is a real constant) is one of the variants of multi-dimensional generalizations of Korteweg--de~Vries equation (KdV) 
$u_t+bu_x+u_{xxx}+uu_x=f(t,x)$.  For the first time it was derived in the three-dimensional case in \cite{ZK} for description of  ion-acoustic waves in magnetized plasma. The equation, considered is the present paper, is known as a model of two-dimensional nonlinear waves in dispersive media propagating in one preassigned ($x$) direction with deformations in the transverse ($y$) direction. A rigorous derivation of the ZK model can be found, for example, in \cite{H-K, LLS}.  

From the point of view of solubility and well-posedness the most significant results for ZK equation and its generalizations were obtained for the initial value problem. In the two-dimensional case the corresponding results in different functional spaces can be found in \cite{S, F89, F95, BL, LP09, LP11, RV, FLP, BJM, GH, MP, FP, G}. For initial-boundary value problems such a theory is most developed for domains, where the variable $y$ is considered in the whole line, (\cite{F02, F07-1, FB08, F08, ST, F12, DL}). 

Initial-boundary value problems posed on domains, where the variable $y$ is considered on a bounded interval, are studied less, although from the physical point of view they seem at least the same important.
Certain technique developed for the case $y\in\mathbb R$ (especially related to the investigation of the corresponding linear equation) up to this moment is extended to the case of bounded $y$ only partially. An initial-boundary value problem in a strip $\mathbb R\times (0,L)$ with periodic boundary conditions was considered in \cite{LPS} for ZK equation and local well-posedness result was established in the spaces $H^s$ for $s>3/2$. This result was improved in \cite{MP} where $s\geq 1$, in addition, in the space $H^1$ appropriate conservation laws provided global well-posedness. 
Initial-boundary value problems in such a strip with homogeneous boundary conditions of different types -- Dirichlet, Neumann or periodic -- were considered in \cite{BF13, F15-2} and results on global well-posedness in classes of weak solutions with power and exponential weights at $+\infty$ were established. Global well-posedness results for ZK equation with certain parabolic regularization also for the initial-boundary value problem in a strip $\mathbb R\times (0,L)$ with homogeneous Dirichlet boundary conditions can be found in \cite{F15-1, F15-2, L15, L16}. 

Similar results on global well-posedness in weighted spaces for initial-boundary value problems in a half-strip $\mathbb R_+\times (0,L)$ were obtained in \cite{LT, L13, F17}. 

Initial-boundary value problems in a bounded rectangle were studied in \cite{STW, DL}. In \cite{STW} either homogeneous Dirichlet or periodic boundary conditions with respect to $y$ were considered and results on global existence and uniqueness of weak solutions were established. In \cite{DL} similar results in more regular classes for homogeneous Dirichlet boundary conditions were obtained. In both papers boundary conditions with respect to $x$ were homogeneous.

In the present paper we consider initial-boundary value problems in a domain $Q_T=(0,T)\times\Omega$, where $\Omega=(0,R)\times (0,L)=\{(x,y): 0<x<R, 0<y<L\}$ is a bounded rectangle of given length $R$ and width $L$, $T>0$ is arbitrary, for equation \eqref{1.1}
with an initial condition
\begin{equation}\label{1.2}
u(0,x,y)=u_0(x,y),\qquad (x,y)\in\Omega,
\end{equation}
boundary conditions for $(t,y)\in B_{T}=(0,T)\times (0,L)$
\begin{equation}\label{1.3}
u(t,0,y)=\mu_0(t,y), \quad u(t,R,y)=\nu_0(t,y),\quad u_x(t,R,y)=\nu_1(t,y)
\end{equation}
and boundary conditions for $(t,x)\in (0,T)\times (0,R)$ of one of the following four types: 
\begin{equation}\label{1.4}
\begin{split}
\mbox{whether}\qquad &a)\mbox{ } u(t,x,0)=u(t,x,L)=0,\\
\mbox{or}\qquad &b)\mbox{ } u_y(t,x,0)=u_y(t,x,L)=0,\\
\mbox{or}\qquad &c)\mbox{ } u(t,x,0)=u_y(t,x,L)=0,\\
\mbox{or}\qquad &d)\mbox{ } u \mbox{ is an  $L$-periodic function with respect to $y$.}
\end{split}
\end{equation}
We use the notation "problem \eqref{1.1}--\eqref{1.4}" for each of these four cases.

The main results consist of theorems on global well-posedness in classes of weak and regular solutions. 
Besides that, certain results on large-time decay of small solutions and boundary controllability, when $\mu_0=\nu_0\equiv 0$, $f\equiv 0$, are established.

In what follows (unless stated otherwise) $j$, $k$, $l$, $m$, $n$ mean non-negative integers, $p\in [1,+\infty]$, $s\in\mathbb R$. Let $[s]$ be the integer part of $s$ ($s-[s]\in [0,1)$). For any multi-index $\alpha=(\alpha_1,\alpha_2)$ let $\partial^\alpha =\partial^{\alpha_1}_{x}\partial^{\alpha_2}_{y}$, let
\begin{equation*}
|D^k\varphi|=\Bigl(\sum_{|\alpha|\leq k}(\partial^\alpha \varphi)^2\Bigr)^{1/2}, \qquad
|D\varphi|=|D^1\varphi|.
\end{equation*}
Let $L_p=L_p(\Omega)$, $W_p^k=W_p^k(\Omega)$, $H^s=H^s(\Omega)$.

Introduce special function spaces taking into account boundary conditions \eqref{1.4}. Let 
$\Sigma= \mathbb R\times (0,L)$, 
$\widetilde{\EuScript S}(\overline{\Sigma})$ be a space of infinitely smooth on $\overline{\Sigma}$ functions $\varphi(x,y)$ such that $\displaystyle{(1+|x|)^n|\partial^\alpha\varphi(x,y)|\leq c(n,\alpha)}$ for any $n$, multi-index $\alpha$, $(x,y)\in \overline{\Sigma}$ and $\partial_y^{2m}\varphi\big|_{y=0} =\partial_y^{2m}\varphi\big|_{y=L}=0$ in the case a), $\partial_y^{2m+1}\varphi\big|_{y=0} =\partial_y^{2m+1}\varphi\big|_{y=L}=0$ in the case b), $\partial_y^{2m}\varphi\big|_{y=0} =\partial_y^{2m+1}\varphi\big|_{y=L}=0$ in the case c), $\partial_y^{m}\varphi\big|_{y=0} =\partial_y^{m}\varphi\big|_{y=L}$ in the case d) for any $m$.

Let $\widetilde H^s(\Sigma)$ be the closure of $\widetilde{\EuScript S}(\overline{\Sigma})$ in the norm $H^s(\Sigma)$ and $\widetilde H^s(I\times (0,L))$ be the restriction of $\widetilde H^s(\Sigma)$ on $I\times (0,L)$ for any interval $I\subset \mathbb R$, $\widetilde H^s =\widetilde H^s(\Omega)$.

It is easy to see, that $\widetilde H^0=L_2$; $\widetilde H^s=H^s$ if $s<0$; for $j \geq 1$ in the case a) $\widetilde H^j=\{\varphi\in H^j: \partial_y^{2m}\varphi|_{y=0}=\partial_y^{2m}\varphi|_{y=L}=0, \ 2m<j\}$, in the case b) $\widetilde H^j=\{\varphi\in H^j: \partial_y^{2m+1}\varphi|_{y=0}=\partial_y^{2m+1}\varphi|_{y=L}=0,\ 2m+1<j\}$,  in the case d) $\widetilde H^j=\{\varphi\in H^j: \partial_y^{m}\varphi|_{y=0}=\partial_y^{m}\varphi|_{y=L}, \ m<j\}$.  

We also use an anisotropic Sobolev space $\widetilde H^{(0,k)}$ which is defined as the restriction on $\Omega$ of a space $\widetilde H^{(0,k)}(\Sigma)$, where the last space is the closure of $\widetilde{\EuScript S}(\overline{\Sigma})$ in the norm $\sum\limits_{m=0}^k \|\partial_y^m\varphi\|_{L_2(\Sigma)}$.

We construct solutions to the considered problems in spaces $X^k(Q_T)$ for $k=0$ and $k=3$, consisting of functions $u(t,x,y)$, such that 
\begin{equation}\label{1.5}
\partial_t^j u\in C([0,T]; \widetilde H^{k-3j})\cap L_2(0,T;\widetilde H^{k-3j+1})
\end{equation}
if $k-3j\geq 0$,
let $X(Q_T)=X^{0}(Q_T)$. 

For description of properties of the boundary data introduce anisotropic functional spaces. Let $B=\mathbb R^t \times (0,L)$. Define the functional space $\widetilde{\EuScript S}(\overline{B})$ similarly to $\widetilde{\EuScript S}(\overline{\Sigma})$, where the variable $x$ is substituted by $t$. Let $\widetilde H^{s/3,s}(B)$ be the closure of $\widetilde{\EuScript S}(\overline{B})$ in the norm $H^{s/3,s}(B)$. 

More exactly, let $\psi_l(y)$, $l=1,2\dots$, be the orthonormal in $L_2(0,L)$ system of the eigenfunctions for the operator $(-\psi'')$ on the segment $[0,L]$ with corresponding boundary conditions   $\psi(0)=\psi(L)=0$ in the case a), $\psi'(0)=\psi'(L)=0$ in the case b), $\psi(0)=\psi'(L)=0$ in the case c), $\psi(0)=\psi(L),\psi'(0)=\psi'(L)$ in the case d), $\lambda_l$ be the corresponding eigenvalues. Such systems are well-known and are written in trigonometric functions.

For any  $\mu\in \widetilde{\EuScript S}(\overline{B})$, $\theta\in\mathbb R$ and $l$ let
\begin{equation}\label{1.6}
\widehat\mu(\theta,l) \equiv \iint_B e^{-i\theta t}\psi_l(y)\mu(t,y)\,dtdy.
\end{equation}
Then the norm in $H^{s/3,s}(B)$ is defined as $\displaystyle\Bigl(\sum\limits_{l=1}^{+\infty} 
\bigl\| (|\theta|^{2/3}+l^2)^{s/2}\widehat\mu(\theta,l)\bigr\|_{L_2(\mathbb R^\theta)}^2\Bigr)^{1/2}$ and the norm in $H^{s/3,s}(I\times (0,L))$ for any interval $I\subset \mathbb R$ as the restriction norm. 

The use of these norm is justified by the following fact. Let $v(t,x,y)$ be the appropriate solution to the initial value problem
$$
v_t+v_{xxx}+v_{xyy}=0,\qquad v\big|_{t=0}=v_0.
$$
Then according to \cite{F08} uniformly with respect to $x\in \mathbb R$
\begin{equation}\label{1.7}
\bigl\|D_t^{1/3}v\bigr\|_{H_{t,y}^{s/3,s}(\mathbb R^2)}^2+
\bigl\|\partial_x v\bigr\|_{H_{t,y}^{s/3,s}(\mathbb R^2)}^2+
\bigl\|\partial_y v\bigr\|_{H_{t,y}^{s/3,s}(\mathbb R^2)}^2
\sim \|v_0\|_{H^s(\mathbb R^2)}^2.
\end{equation}
%(here $D^\alpha$ denotes the Riesz potential of the order $-\alpha$).

Introduce the notion of weak solutions to the considered problems.

\begin{definition}\label{D1.1}
Let $u_0\in L_2$, $\mu_0,\nu_0,\nu_1\in L_2(B_T)$, $f\in L_1(0,T;L_2)$. A function $u\in L_\infty(0,T;L_2)$ is called a generalized solution to problem \eqref{1.1}--\eqref{1.4} if for any function $\phi\in L_2(0,T;\widetilde H^2)$, such that $\phi_t, \phi_{xxx}, \phi_{xyy}\in L_2(Q_T)$, $\phi\big|_{t=T}\equiv 0$, $\phi\big|_{x=0} =\phi_x\big|_{x=0} =\phi\big|_{x=R}\equiv 0$, the following equality holds:
\begin{multline}\label{1.8}
\iiint_{Q_T}\Bigl[u(\phi_t+b\phi_x+\phi_{xxx}+\phi_{xyy}) +\frac 12 u^2 \phi_x +f\phi\Bigr]\,dxdydt 
+\iint_{\Omega} u_0\phi\big|_{t=0}\,dxdy \\+
\iint_{B_T} \Bigl[\mu_0\phi_{xx}\big|_{x=0}-\nu_0\phi_{xx}\big|_{x=R}+
\nu_1\phi_x\big|_{x=R}\Bigr]\,dydt =0.
\end{multline}
\end{definition}

\begin{remark}\label{R1.1}
Note that the integrals in \eqref{1.8} are well defined (in particular, since $\phi_x\in L_2(0,T;H^2)\subset L_2(0,T;L_\infty)$).
\end{remark}

Now we can formulate the main results of the paper concerning  well-posedness, which means existence, uniqueness of solutions and Lipschitz continuity of the map $(u_0,\mu_0,\nu_0,\nu_1,f)\mapsto u$ in the corresponding norms on any ball in the space of the input data.

\begin{theorem}\label{T1.1}
Let $u_0\in L_2$, $f\in L_1(0,T; L_2)$ for certain $T>0$, $\mu_0,\nu_0\in \widetilde H^{s/3,s}(B_T)$ for certain $s>3/2$, $\nu_1\in L_2(B_T)$. Then problem \eqref{1.1}--\eqref{1.4} is well-posed in the space $X(Q_T)$.
\end{theorem}

\begin{remark}\label{R1.2}
In the cases a) and d) for $\mu_0=\nu_0=\nu_1\equiv 0$ similar result was established in \cite{STW}. In the last paper certain properties of traces of $u_x$ with respect to $x$ were also obtained.
\end{remark}

\begin{theorem}\label{T1.2}
Let $u_0\in \widetilde H^3$, $f\in C([0,T];L_2)\cap L_2(0,T; \widetilde H^{(0,2)})$, $f_t\in L_1(0,T;H^{-1})$ for certain $T>0$, $\mu_0,\nu_0\in \widetilde H^{4/3,4}(B_T)$, $\nu_1\in \widetilde H^{1,3}(B_T)$, $\mu_0(0,y)\equiv u_0(0,y)$, $\nu_0(0,y)\equiv u_0(R,y)$, $\nu_1(0,y)\equiv u_{0x}(R,y)$. Then problem \eqref{1.1}--\eqref{1.4} is well-posed in the space $X^3(Q_T)$.
\end{theorem}

\begin{remark}\label{R1.3}
According to \eqref{1.7} the assumptions on the boundary data $\mu$ are natural. In \cite{DL} for construction of regular solutions only homogeneous Dirichlet boundary conditions were considered. Moreover, in that paper for $u_{yyy}$ was established only that $u_{yyy}\in L_2(Q_T)$.  
\end{remark}

Estimates on solutions, established in the proof of Theorem~\ref{T1.1}, provide the following result on the large-time decay of small solutions. Let $B_+ = \mathbb R_+^t\times (0,L)$.

\begin{theorem}\label{T1.3}
Let there exists $\delta\in (0,1)$ such that $\varkappa>0$, where
\begin{equation}\label{1.9}
\varkappa = -b+\left\{
\begin{aligned}
\pi^2(1-\delta)\bigl(\frac 3{R^2} + \frac 1{L^2}\bigr) \qquad&\mbox{in the case a)},\\
\pi^2(1-\delta)\bigl(\frac 3{R^2} + \frac 1{4L^2}\bigr) \qquad&\mbox{in the case c)},\\
\pi^2(1-\delta)\frac 3{R^2} \qquad&\mbox{in the cases b) and d)}.
\end{aligned}
\right.
\end{equation}
Let
\begin{equation}\label{1.10}
\epsilon_0 = \frac{3^{5/4}\pi\delta}{4} \times\left\{
\begin{aligned}
\max\Bigl(\frac{\sqrt{3}}{R},\frac 1{L}\Bigr) \qquad&\mbox{in the case a)},\\
\max\Bigl(\frac{\sqrt{3}}{R},\frac 1{2L}\Bigr)\qquad&\mbox{in the case c)},\\
\frac{\sqrt{3}}{R}\times\frac {3^{1/4}(\pi L)^{1/2}}{R^{1/2}+3^{1/4}(\pi L)^{1/2}} \qquad&\mbox{in the cases b) and d)}.
\end{aligned}
\right.
\end{equation}
Let $u_0\in L_2$, $\nu_1\in L_2(B_+)$, 
$$
\|u_0\|^2_{L_2}+\|\nu_1\|^2_{L_2(B_+)}\leq \epsilon_0^2,
$$
$f\equiv 0$, $\mu_0=\nu_0\equiv 0$. Then the corresponding unique weak solution $u(t,x,y$) to problem \eqref{1.1}--\eqref{1.4} from the space $X(Q_T)$ $\forall T>0$ satisfies an inequality
\begin{equation}\label{1.11}
\|u(t,\cdot,\cdot)\|^2_{L_2}\leq(1+R) e^{-\varkappa t/(1+R)}\Bigl[\|u_0\|_{L_2}^2 +
\bigl\|e^{\varkappa\tau/(2(1+R))}\nu_1\bigr\|^2_{L_2(B_t)}\Bigr]
\quad \forall t\geq 0.
\end{equation}
\end{theorem}

\begin{remark}\label{R1.4}
In the case a) if $b=1$, $\nu_1\equiv 0$ a similar result for regular solutions in a slightly different form was previously established in \cite{DL}.
\end{remark}

On the basis of ideas and results from \cite{R} as an application of the developed technique we obtain the following result on the controllability problem for system \eqref{1.4}--\eqref{1.4}  with the unknown boundary control $\nu_1$ and with the condition of final overdetermination
\begin{equation}\label{1.12}
u(T,x,y)=u_T(x,y),\qquad (x,y)\in\Omega.
\end{equation}

\begin{theorem}\label{T1.4}
Let for any natural $l$, such that $\lambda_l <b$ (where $\lambda_l$ are the aforementioned eigenvalues of the operator $(-\psi'')$ on $(0,L)$ with corresponding boundary conditions),
\begin{equation}\label{1.13}
R \ne 2\pi \Bigl(\frac{k^2+km+m^2}{3(b-\lambda_l)}\Bigr)^{1/2}\qquad \forall k,m\in\mathbb N.
\end{equation}
Let $T>0$, $f\equiv 0$, $\mu_0=\nu_0\equiv 0$, $u_0, u_T\in L_2$. Then there exists $\varepsilon>0$, such that if $\|u_0\|_{L_2}, \|u_T\|_{L_2} <\varepsilon$ there exists a function $\nu_1\in L_2(B_T)$, such that there exists a unique solution $u\in X(Q_T)$ to problem \eqref{1.1}--\eqref{1.4}, satisfying \eqref{1.12}.
\end{theorem}

\begin{remark}\label{R1.5}
In comparison with Theorem~\ref{T1.3} the constant $\varepsilon$ is not evaluated explicitly.
\end{remark}

Further, let $\eta(x)$ denotes a cut-off function, namely, $\eta$ is an infinitely smooth non-decreasing function on $\mathbb R$  such that $\eta(x)=0$ when $x\leq 0$, $\eta(x)=1$ when $x\geq 1$, $\eta(x)+\eta(1-x)\equiv 1$.

We drop limits of integration in integrals over the rectangle $\Omega$.

\medskip

The following interpolating inequality specifying the one from \cite{LSU} is crucial for the study.

\begin{lemma}\label{L1.1}
Let $\varphi(x,y)\in H^1$ satisfy $\varphi\big|_{x=0}=0$ or $\varphi|_{x=R}=0$, then the following inequalities hold:
\begin{multline}\label{1.14}
\iint \varphi^4 dxdy \leq 4\Bigl(\iint \varphi_x^2\,dxdy \iint \varphi_y^2\,dxdy\Bigr)^{1/2} \iint \varphi^2\,dxdy \\
+\frac{4\sigma}{L} \Bigl(\iint \varphi_x^2\,dxdy\Bigr)^{1/2} \Bigl(\iint \varphi^2\,dxdy\Bigr)^{3/2},
\end{multline}
\begin{multline}\label{1.15}
\iint |\varphi|^3 dxdy \leq 2\Bigl(\iint \varphi_x^2\,dxdy \iint \varphi_y^2\,dxdy\Bigr)^{1/4} \iint \varphi^2\,dxdy \\
+\frac{2\sigma}{L^{1/2}} \Bigl(\iint \varphi_x^2\,dxdy\Bigr)^{1/4} \Bigl(\iint \varphi^2\,dxdy\Bigr)^{5/4},
\end{multline}
where $\sigma=0$ if $\varphi\big|_{y=0}=0$ or $\varphi\big|_{y=L}=0$ and $\sigma=1$ in the general case.
\end{lemma}

\begin{proof}
We follow the argument from \cite{LSU} and start with the following inequality:
\begin{equation}\label{1.16}
\iint \varphi^2\,dxdy \leq \iint |\varphi_x|\,dxdy \Bigl( \iint |\varphi_y|\,dxdy +
\frac{2\sigma}{L} \iint |\varphi|\,dxdy\Bigr).
\end{equation}
In fact,
$$
\sup\limits_{x\in (0,R)} |\varphi(x,y)| \leq \int_0^R |\varphi_x(x,y)|\,dx;
$$
in the general case $\varphi(x,y) = \displaystyle \varphi(x,y)\frac{y}{L} + \varphi(x,y)\frac{L-y}{L} \equiv
\varphi_1(x,y)+\varphi_2(x,y)$, where
$$
\sup\limits_{y\in (0,L)} |\varphi_j(x,y)| \leq \int_0^L |\varphi_y(x,y)|\alpha_j(y)\,dy +\frac{1}{L}\int_0^L |\varphi(x,y)|\,dy,
$$
 where either $\alpha_j(y)\equiv y/L$, or $\alpha_j(y)\equiv (L-y)/L$, therefore,
$$
\sup\limits_{y\in (0,L)} |\varphi(x,y)| \leq \int_0^L |\varphi_y(x,y)|\,dy + \frac{2\sigma}{L} \int_0^L |\varphi(x,y)|\,dy.
$$
Since
$$
\iint \varphi^2(x,y)\,dxdy 
\leq \int_0^L \sup_{x\in (0,R)} |\varphi(x,y)|\,dy
\int_0^R \sup_{y\in (0,L)} |\varphi(x,y)|\,dx,
$$
we obtain \eqref{1.16}. Therefore,
$$
\iint \varphi^4\,dxdy \leq \iint \bigl|(\varphi^2)_x\bigr|\,dxdy \Bigl(
\iint \bigl|(\varphi^2)_y\bigr|\,dxdy + \frac{2\sigma}{L}\iint \varphi^2\,dxdy\Bigr),
$$
whence \eqref{1.14} succeeds. Inequality \eqref{1.15} obviously follows from \eqref{1.14} and H\"older's inequality.
\end{proof}

For the decay results, we need Steklov's inequalities in the following form: for $\psi\in H_0^1(0,L)$,
\begin{equation}\label{1.17}
\int_0^L \psi^2(y)\,dy \leq \frac{L^2}{\pi^2} \int_0^L \bigl(\psi'(y)\bigr)^2\,dy,
\end{equation}
for $\psi\in H^1(0,L)$, $\psi\big|_{y=0}=0$,
\begin{equation}\label{1.18}
\int_0^L \psi^2(y)\,dy \leq \frac{4L^2}{\pi^2} \int_0^L \bigl(\psi'(y)\bigr)^2\,dy.
\end{equation}

In the following obvious interpolating results values of constants are indifferent for our purposes: for $\varphi\in H^1$
\begin{equation}\label{1.19}
\sup\limits_{x\in [0,R]}\int_0^L \varphi^2(x,y)\,dy \leq c \Bigl(\iint \varphi_x^2\,dxdy \iint \varphi^2\,dxdy\Bigr)^{1/2}+ c\iint \varphi^2\,dxdy,
\end{equation}
\begin{equation}\label{1.20}
\|\varphi\|_{L_4} \leq c\|\varphi\|^{1/2}_{H^1}\|\varphi\|^{1/2}_{L_2}
\end{equation}
and for $\varphi\in H^2$
\begin{equation}\label{1.21}
\|\varphi\|_{L_{\infty}} \leq c \|\varphi\|_{H^2}.
\end{equation}

\begin{lemma}\label{L1.2}
For $k=1$ and $k=2$ introduce functional spaces 
$$
H^{(-k,0)}= \{\varphi = \sum\limits_{m=0}^k \partial_x^m \varphi_m: \varphi_m \in L_2\}
$$
endowed with the natural norms. Then for $j=1$ and $j=2$
\begin{equation}\label{1.22}
\|\partial_x^j \varphi\|_{L_{2}} \leq c(R)\bigl( \|\varphi_{xxx}\|_{H^{(j-3,0)}} +
\|\varphi\|_{L_2}\bigr).
\end{equation}
\end{lemma}

\begin{proof}
First consider the case $j=2$. For any $\psi\in L_2$ let $a_0(y)\equiv \displaystyle \int_0^R \psi(x,y)\,dx$, then $\|a_0\|_{L_2(0,L)} \leq c\|\varphi\|_{L_2}$. Let $\omega(x)\in C_0^\infty(0,R)$, $\|\omega\|_{L_2(0,R)}=1$.

Define $\displaystyle\psi_0(x,y)\equiv \int_0^x \psi(z,y)\,dz -a_0(y)\omega(x)$, then $\|\psi_0\|_{L_2}, \|\psi_{0x}\|_{L_2}\leq c\|\psi\|_{L_2}$, $\psi_0\big|_{x=0}=\psi_0\big|_{x=R}=0$, $\psi = \psi_{0x} +a_0\omega'$. We have:
\begin{multline*}
\langle \varphi_{xx},\psi_{0x}\rangle = - \langle \varphi_{xxx},\psi_0\rangle \leq
\|\varphi_{xxx}\|_{H^{(-1,0)}}\bigl(\|\psi_0\|_{L_2} +\|\psi_{0x}\|_{L_2}\bigr) \\
\leq c \|\varphi_{xxx}\|_{H^{(-1,0)}}\|\psi\|_{L_2},
\end{multline*}
$$
\langle \varphi_{xx}, a_0\omega'\rangle = \langle \varphi, a_0\omega'''\rangle \leq 
c\|\varphi\|_{L_2}\|\psi\|_{L_2}.
$$
Therefore,
$$
\langle \varphi_{xx},\psi\rangle \leq c \bigl( \|\varphi_{xxx}\|_{H^{(-1,0)}} +
\|\varphi\|_{L_2}\bigr)\|\psi\|_{L_2}
$$
and \eqref{1.22} for $j=2$ follows.

Now let $j=1$. For $\psi\in L_2$ define $\displaystyle a_1(y)\equiv \int_0^R \psi_0(x,y)\,dx$, $\displaystyle \psi_1(x,y) \equiv \int_0^x \psi_0(z,y)\,dz -a_1(y)\omega(x)$. Then $\psi =\psi_{1xx}+a_0\omega' +a_1\omega''$ and similarly to the previous case
\begin{multline*}
\langle \varphi_{x},\psi_{1xx}\rangle = \langle \varphi_{xxx},\psi_1\rangle \leq
\|\varphi_{xxx}\|_{H^{(-2,0)}}\bigl(\|\psi_1\|_{L_2} +\|\psi_{1x}\|_{L_2}+\|\psi_{1xx}\|_{L_2}\bigr) \\
\leq c \|\varphi_{xxx}\|_{H^{(-2,0)}}\|\psi\|_{L_2},
\end{multline*}
$$
\langle \varphi_{x}, a_0\omega'+a_1\omega''\rangle =- \langle \varphi, a_0\omega''+a_1\omega'''\rangle \leq 
c\|\varphi\|_{L_2}\|\psi\|_{L_2}.
$$
Therefore,
$$
\langle \varphi_{x},\psi\rangle \leq c \bigl( \|\varphi_{xxx}\|_{H^{(-2,0)}} +
\|\varphi\|_{L_2}\bigr)\|\psi\|_{L_2},
$$
which finishes the proof.
\end{proof}

The paper is organized as follows. Auxiliary linear problems are considered in Section~\ref{S2}.  Section~\ref{S3} is devoted to the well-posedness results for the original problems. Decay of solutions is studied in Section~\ref{S4} and boundary controllability in Section~\ref{S5}.

\section{Auxiliary linear problems}\label{S2}

Consider a linear equation
\begin{equation}\label{2.1}
u_t+bu_x+u_{xxx}+u_{xyy}=f(t,x,y).
\end{equation}

For any interval $I\subset\mathbb R^x$ and $k$ introduce functional spaces 
\begin{multline*}
Y_k((0,T)\times I\times (0,L))
=\{u(t,x,y): \partial_t^j u\in C([0,T];\widetilde H^{k-3j}(I\times (0,L)),\quad\text{if}\ j\leq k/3,\\
\partial_x^n u\in C_b(\overline{I};\widetilde H^{(k-n+1)/3,k-n+1}(B_T)),\quad \text{if}\ n\leq k+1\}
\end{multline*}
(here and further the lower index 'b" means a bounded map),
\begin{multline*}
M_k((0,T)\times I\times (0,L))=\{f(t,x,y): \partial_t^j f\in L_2(0,T;\widetilde H^{k-3j}(I\times(0,L)),\\ \text{if}\ 
j\leq j_0=[(k+1)/3]\}.
\end{multline*}
Let $\widetilde \Phi_0(x,y) \equiv u_0(x,y)$
and for $j\geq 1$
$$
\widetilde\Phi_j(x,y) \equiv \partial^{j-1}_t f(0,x,y)-
(b\partial_x+\partial_x^3+\partial_x\partial_y^2)\widetilde\Phi_{j-1}(x,y). 
$$

Solutions to an initial-boundary value problem in a domain
$\Pi_T=(0,T)\times \Sigma$ with the initial profile
\eqref{1.2} for $(x,y)\in \Sigma$ and boundary conditions \eqref{1.4} for $(t,x)\in (0,T)\times\mathbb R$ for equation \eqref{2.1} can be constructed in a form (see \cite{F17})
\begin{equation}\label{2.2}
u(t,x,y)=S(t,x,y;u_0)+K(t,x,y;f),
\end{equation}
where potentials $S$ and $K$ are given by formulas
\begin{equation} \label{2.3}
\begin{gathered}
S(t,x,y;u_0) \equiv \sum_{l=1}^{+\infty} \frac{1}{2\pi}
\int_{\mathbb R}\, e^{it(\xi^3-b\xi+\lambda_l\xi)} e^{i\xi x}\widehat{u}_0(\xi,l)\,d\xi \psi_l(y),\\
K(t,x,y;f)\equiv\int^t_0 S(t-\tau,x,y;f(\tau,\cdot,\cdot))\,d\tau,
\end{gathered}
\end{equation}
where the functions $\widehat{u}_0(\xi,l)$ are defined similarly to \eqref{1.6}.

\begin{lemma}\label{L2.1}
If $u_0\in \widetilde H^k(\Sigma)$, $f\in M_k(\Pi_T)$ for some $T>0$
and $k\geq 0$, then a unique solution $u(t,x,y)\in Y_k(\Pi_T)$ to problem
\eqref{2.1}, \eqref{1.2}, \eqref{1.4} exists and for any $t_0\in (0,T]$
\begin{equation} \label{2.4}
\begin{aligned}
&\|u\|_{Y_k(\Pi_{t_0})}\\
&\leq c(T,k,b)\Bigl(\|u_0\|_{\widetilde H^k(\Sigma)}+
t_0^{1/6}\|f\|_{M_k(\Pi_{t_0})}
 +\sum_{j=0}^{j_0-1} \|\partial_t^j
f\big|_{t=0}\|_{\widetilde H^{k-3(j+1)}(\Sigma)}\Bigr).
\end{aligned}
\end{equation}
\end{lemma}

\begin{proof}
First of all note that uniqueness of solutions to the considered problem in the space $L_2(\Pi_T)$ (in fact, in a more wide class) was established in \cite{BF13}. Next, note that
\begin{equation}\label{2.5}
\partial_t^j S(t,x,y;u_0)+\partial_t^j K(t,x,y;f)=
S(t,x,y;\widetilde\Phi_j)+ K(t,x,y;\partial_t^j f).
\end{equation}
Then the corresponding estimates on $\partial_t^j u$ in the norm $C([0,t_0];\widetilde H^{k-3j}(\Sigma))$ by
$\|\widetilde\Phi_j\|_{\widetilde H^{k-3j}(\Sigma)}$ and $\|\partial_t^j f\|_{L_1(0,t_0;\widetilde H^{k-3j}(\Sigma))}$ easily follow. In turn,
\begin{equation}\label{2.6}
\|\widetilde\Phi_j\|_{\widetilde H^{k-3j}(\Sigma)} \leq c(k,b)\Bigl( \|u_0\|_{\widetilde H^k(\Sigma)} +
\sum_{m=0}^{j-1} \|\partial_t^m
f\big|_{t=0}\|_{\widetilde H^{k-3(m+1)}(\Sigma)}\Bigr).
\end{equation}
It was proved in \cite{F17} that for $s\in [0,3]$ 
\begin{equation}\label{2.7}
\|u\|_{C_b(\mathbb{R};\widetilde H^{s/3,s}(B_{t_0}))}
\leq c(T,b)\Bigl(\|u_0\|_{\widetilde H^{s-1}(\Sigma)}+t_0^{1/2-s/6} \|f\|_{L_2(0,t_0;\widetilde H^{s-1}(\Sigma))}\Bigr).
\end{equation}
Applying \eqref{2.5}--\eqref{2.7} for $j=[(k+1-n-l)/3]\leq j_0$, $s=k+1-n-l-3j\in [0,3)$, we derive that
\begin{multline}\label{2.8}
\|\partial_t^j\partial_x^n\partial_y^l u\|_{C_b(\mathbb{R};\widetilde H^{s/3,s}(B_{t_0}))}   \leq
\|S(\cdot,\cdot,\cdot;\partial^n_x\partial^l_y\widetilde\Phi_j)\|_{C_b(\mathbb{R};\widetilde H^{s/3,s}(B_{t_0}))} \\
+\|K(\cdot,\cdot,\cdot;\partial^n_x\partial^l_y\partial_t^j f)\|_{C_b(\mathbb{R};\widetilde H^{s/3,s}(B_{t_0}))}  \leq c(T,k,b)\Bigl(\|u_0\|_{\widetilde H^k(\Sigma)} \\+
\sum_{m=0}^{j-1} \|\partial_t^m f\big|_{t=0}\|_{\widetilde H^{k-3(m+1)}(\Sigma)} + 
t_0^{1/2-s/6}\|\partial_t^j f\|_{L_2(0,t_0;\widetilde H^{k-3j})}\Bigr). 
\end{multline} 
Finally, it is suffice to note that the
minimal value $1/6$ for the degree $(1/2-s/6)$ in \eqref{2.8} is achieved if $k+1-n-l=3j+2$.
\end{proof}

Next, consider an initial-boundary value problem in a domain $\Pi_T^-=(0,T)\times \Sigma_-$, $\Sigma_-=\mathbb R_-\times (0,L) =\{(x,y): x<0, 0<y<L\}$, for equation \eqref{2.1} with initial condition \eqref{1.2} for $(x,y)\in \Sigma_-$, boundary conditions \eqref{1.4} for $(t,x)\in (0,T)\times \mathbb R_-$ and
\begin{equation}\label{2.9}
u(t,0,y)=\nu_0(t,y),\quad u_x(t,0,y)=\nu_1(t,y), \quad (t,y)\in B_T.
\end{equation}
Weak solutions to this problem are understood similarly to Definition~\ref{D1.1} with obvious changes, moreover, due to the absence of nonlinearity one can take solutions from the space $L_2(\Pi_T^-)$.

\begin{lemma}\label{L2.2}
A generalized solution to problem \eqref{2.1}, \eqref{1.2}, \eqref{1.4}, \eqref{2.9} is unique in the space $L_2(\Pi_T^-)$.
\end{lemma}

\begin{proof}
According to \cite{F17} the backward problem in $\Pi_T^-$ for equation \eqref{2.1} with boundary conditions $u\big|_{t=T}=0$, $u\big|_{x=0}=0$ and \eqref{1.4} for $f\in C_0^\infty(\Pi_T^-)$ has a solution $u\in C([0,T];\widetilde H^3(\Sigma_-))$, $u_t\in C([0,T];L_2(\Sigma_-))$, therefore, the desired result is obtained via the standard 
H\"olmgren's argument.
\end{proof}

\begin{lemma}\label{L2.3}             
Let $u_0\equiv 0$, $\nu_0, \nu_1\in C_0^\infty(B_+)$, $f\equiv 0$. Then there exists a solution $u(t,x,y)$ to problem \eqref{2.1}, \eqref{1.2}, \eqref{1.4}, \eqref{2.9} such that $\partial^j_t u \in C_b(\overline{\mathbb R}_+^t; \widetilde H^n(\Sigma_-))$ for any $j$ and $n$.
\end{lemma}

\begin{proof}
Let $v(t,x,y)\equiv u(t,x,y) -\nu_0(t,y)\eta(x+1)-\nu_1(t,y)x\eta(x+1)$, then the original problem is equivalent to the problem of \eqref{2.1}, \eqref{1.2}, \eqref{1.4}, \eqref{2.9} type for the function $v$ with homogeneous initial-boundary conditions and $f\equiv -\nu_{0t}\eta-\nu_{1t}x\eta-b\nu_0\eta'-b\nu_1(x\eta)' -\nu_0\eta'''-\nu_1(x\eta)''' -\nu_{0yy}\eta'-\nu_{1yy}(x\eta)'$. 

Let $\{\varphi_j(x): j=1,2,\dots\}$ be a set of linearly independent functions complete in the space $\{\varphi \in H^3(\mathbb R_-): \varphi(0)=0\}$. We use the Galerkin method and seek an approximate solution in a form $v_k(t,x,y)= \sum\limits_{j,l=1}^k c_{kjl}(t) \varphi_j(x)\psi_l(y)$ (remind that $\psi_l$ are the orthonormal in $L_2(0,L)$ eigenfunctions for the operator $(-\psi'')$ on the segment $[0,L]$ with corresponding boundary conditions) via conditions for $i,m=1,\dots,k$, $t\in [0,T]$
\begin{equation}\label{2.10}
\iint_{\Sigma_-} \bigl(v_{kt}\varphi_i(x)\psi_m(y) -v_k(b\varphi_i'\psi_m+\varphi'''_i\psi_m +\varphi'_i\psi''_m)\bigr)\,dxdy - \iint_{\Sigma_-} f\varphi_i\psi_m\,dxdy=0,
\end{equation}
$c_{kjl}(0)=0$. In particular, $v_k\big|_{t=0}=0$. Moreover, putting in \eqref{2.10} $t=0$, multiplying by $c'_{kim}(0)$ and summing with respect to $i,m$, we obtain that $v_{kt}\big|_{t=0}=0$. Next, differentiating \eqref{2.10} $j$ times with respect to $t$ we derive that
\begin{multline}\label{2.11}
\iint_{\Sigma_-} \bigl(\partial_t^{j+1}v_{k}\varphi_i\psi_m -\partial_t^j v_k(b\varphi_i'\psi_m+\varphi'''_i\psi_m +\varphi'_i\psi''_m)\bigr)\,dxdy \\- \iint_{\Sigma_-} \partial_t^j f\varphi_i\psi_m\,dxdy=0.
\end{multline}
Then by induction with respect to $j$ we find that $\partial_t^j v_k\big|_{t=0} =0$ for all $j$. Since $\psi_m^{(2n)}(y)=(-\lambda_m)^n\psi_m(y)$ it follows from \eqref{2.10} and \eqref{2.11} that for all $j$ and $n$
\begin{multline}\label{2.12}
\iint_{\Sigma_-} \bigl(\partial_t^{j+1}\partial_y^{n} v_{k}\varphi_i\psi_m^{(n)} -
\partial_t^j\partial_y^{n} v_k(b\varphi'_m\psi_m^{(n)}+ \varphi'''_i\psi_m^{(n)}+\varphi'_i\psi_m^{(n+2)})\bigr)\,dxdy \\- 
\iint_{\Sigma_-} \partial_t^j\partial_y^{n}f\varphi_i\psi_m^{(n)}\,dxdy=0.
\end{multline}
Multiplying \eqref{2.12} by $2c^{(j)}_{kim}(t)$ and summing with respect to $i,m$, we find that
\begin{equation}\label{2.13}
\frac d{dt} \iint_{\Sigma_-} (\partial_t^j\partial_y^n v_k)^2\,dxdy +\int_0^L (\partial_t^j\partial_y^n v_{kx})^2\big|_{x=0}\,dy = 
2\iint_{\Sigma_-} \partial_t^j\partial_y^nf \partial_t^j\partial_y^nv_k\,dxdy,
\end{equation}
and, therefore, for all $j$ and $n$
\begin{equation}\label{2.14}
\|\partial_t^jv_k\|_{L_\infty(\mathbb R_+^t;\widetilde H^{(0,n)}(\Sigma_-))} \leq \|\partial_t^j f\|_{L_1(\mathbb R_+^t;\widetilde H^{(0,n)}(\Sigma_-))}.
\end{equation}
Estimate \eqref{2.14} provide existence of a weak solution $v(t,x,y)$ to the considered problem such that $\partial_t^j v \in C_b(\overline{\mathbb R}_+^t;\widetilde H^{(0,n)}(\Sigma_-))\ \forall n, j$ in the following sense: for any $T>0$ and a function $\phi\in L_2(0,T;\widetilde H^2(\Sigma_-))$, such that $\phi_t, \phi_{xxx}, \phi_{xyy}\in L_2(\Pi_T^-)$, $\phi\big|_{t=T}=0$, $\phi\big|_{x=0}=0$, the following equality holds:
\begin{equation}\label{2.15}
\iiint_{\Pi_T^-}\Bigl[v(\phi_t+b\phi_x+\phi_{xxx}+\phi_{xyy}) +f\phi\Bigr]\,dxdydt =0.
\end{equation}
Note, that the traces of the function $v$ satisfy zero condition \eqref{1.2} and condition \eqref{1.4}. Moreover, it follows from \eqref{2.15} that $\partial_t^j\partial_y^n v_{xxx} \in C_b(\overline{\mathbb R}_+^t;H^{(-1,0)}(\Sigma_-))\ \forall n,j$, therefore, $\partial_t^j\partial_y^n v_x \in C_b(\overline{\mathbb R}_+^t;L_2(\Sigma_-))\ \forall n,j$ (see \cite{F17}) and one more application of \eqref{2.15} yields that $\partial_t^j v_{xxx} \in C_b(\overline{\mathbb R}_-^t;\widetilde H^{(0,n)}(\Sigma_-))\ \forall n,$, the function $v$ satisfies the corresponding equation \eqref{2.1} a.e. in $\Pi_T^+$ and its traces satisfy zero conditions \eqref{2.9}. Finally, with the use of induction with respect to $m$ one can find that $\partial^j_t\partial_x^{3m} v \in C_b(\overline{\mathbb R}_+^t; \widetilde H^{(0,n)})$ for all $m,j,n$.
\end{proof}

In what follows, we need some properties of solutions to an algebraic equation
\begin{equation}\label{2.16}
z^3 +az+p=0, \qquad a\in \mathbb R,\quad p=\varepsilon+i\theta \in \mathbb C.
\end{equation}
For $\varepsilon>0$ we denote by $z_1(p,a)$ and $z_2(p,a)$ two roots of this equation with positive real parts (the rest root has the negative real part). Let
\begin{equation}\label{2.17}
r_j(\theta,a) = \lim\limits_{\varepsilon\to +0} z_j(\varepsilon+i\theta),\quad j=1, 2.
\end{equation}
The values $r_j(\theta,a)$ are roots of the equation
\begin{equation}\label{2.18}
r^3+ar+i\theta=0
\end{equation}
and $\Re r_j\geq 0$, $j=1$ and $2$. Moreover, it can be shown with the use of the Cardano formula, that for certain positive constants $c_0$, $c_1$ and all $\theta$ and $a$
\begin{gather}\label{2.19}
|r_j(\theta,a)| \leq c_1(|\theta|^{1/3}+|a|^{1/2}),\quad j=1,2,\\
\label{2.20}
|r_1(\theta,a)-r_2(\theta,a)| \geq  c_0(|\theta|^{1/3}+|a|^{1/2})
\end{gather}
(for more details see, for example, \cite{F07-2}).

Now introduce special solutions of equation \eqref{2.1} for $f\equiv 0$ of "boundary potential" type.

\begin{definition}\label{D2.1}
Let $\nu \in\widetilde{\EuScript S}(\overline{B})$. Define for $x\leq 0$
\begin{gather}\label{2.21}
J_0(t,x,y;\nu) \equiv \sum\limits_{l=1}^{+\infty} \EuScript F^{-1}_t \Bigl[\frac{r_1 e^{r_2x}-r_2e^{r_1x}}{r_1-r_2}\widehat\nu(\theta,l)\Bigr](t) \psi_l(y),\\
\label{2.22}
J_1(t,x,y;\nu) \equiv \sum\limits_{l=1}^{+\infty} \EuScript F^{-1}_t \Bigl[\frac{e^{r_1x}-e^{r_2x}}{r_1-r_2}\widehat\nu(\theta,l)\Bigr](t) \psi_l(y),
\end{gather}
where $\widehat\nu(\theta,l)$ is given by formula \eqref{1.6} and $r_j=r_j(\theta,b-\lambda_l)$ -- by formula \eqref{2.17}.
\end{definition}

\begin{lemma}\label{L2.4}
For any $s\in\mathbb R$ the notion of the function $J_0(t,x,y;\nu)$ can be extended by continuity in the space $C_b(\overline{\mathbb R}^x_-;\widetilde H^{s/3,s}(B))$ to any function $\nu\in \widetilde H^{s/3,s}(B)$. Moreover, for any $n$
\begin{equation}\label{2.23}
\|\partial_x^n J_0(\cdot,\cdot,\cdot;\nu)\|_{ C_b(\overline{\mathbb R}_-^x;\widetilde H^{(s-n)/3.s-n}(B))} \leq 
c(n,b)\|\nu\|_{\widetilde H^{s/3.s}(B)}
\end{equation}
and $J_0\big|_{x=0}=\nu$, $J_{0x}\big|_{x=0}=0$.
\end{lemma}

\begin{proof}
Since 
$$
\partial_x^n \widehat J_0(\theta,x,l;\nu) =   \frac{r_1r_2^n e^{r_2x}-r_2r_1^ne^{r_1x}}{r_1-r_2}\widehat\nu(\theta,l)
$$
and $\Re (r_jx)\leq 0$ the assertion of the lemma follows from \eqref{2.19}, \eqref{2.20}.
\end{proof}

\begin{lemma}\label{L2.5}
For any $s\in\mathbb R$ and $R>0$ the notion of the function $J_1(t,x,y;\nu)$ can be extended by continuity in the space $C([-R,0];\widetilde H^{s/3,s}(B))$ to any function $\nu\in \widetilde H^{s/3,s}(B)$. Moreover, 
\begin{equation}\label{2.24}
\|x^{-1}J_1(\cdot,\cdot,\cdot;\nu)\|_{ C_b(\overline{\mathbb R}_-^x;\widetilde H^{s/3.s}(B))} \leq 
\|\nu\|_{\widetilde H^{s/3.s}(B)},
\end{equation}
for any $n\geq 1$
\begin{equation}\label{2.25}
\|\partial_x^n J_1(\cdot,\cdot,\cdot;\nu)\|_{ C_b(\overline{\mathbb R}_-^x;\widetilde H^{(s-n+1)/3.s-n+1}(B))} \leq  c(n,b)\|\nu\|_{\widetilde H^{s/3.s}(B)}
\end{equation}
and $J_1\big|_{x=0}=0$, $J_{1x}\big|_{x=0}=\nu$.
\end{lemma}

\begin{proof}
Since 
$$
\partial_x^n \widehat J_1(\theta,x,l;\nu) =   \frac{r_1^n e^{r_1x}-r_2^ne^{r_2x}}{r_1-r_2}\widehat\nu(\theta,l)
$$
and $\Re (r_jx)\leq 0$ (in particular, $|\widehat J_1(\theta,x,l;\nu)|\leq |x\widehat\nu(\theta,l)|$) the assertion of the lemma follows from \eqref{2.19}, \eqref{2.20}.
\end{proof}

\begin{remark}\label{R2.1}
In the most important for us case $s\geq 0$ the values $\widehat\nu(\theta,l)$ can be defined directly as limits in $L_2(B)$, for example, of integrals $\displaystyle \int_{-T}^T\!\int_0^L e^{-i\theta t}\psi_l(y)\nu(t,y)\,dtdy$, $T\to +\infty$. Then the functions $J_0(t,x,y;\nu)$ and $J_1(t,x,y;\nu)$ can be equivalently defined simply by formulas \eqref{2.21}, \eqref{2.22}.
\end{remark}

\begin{lemma}\label{L2.6}
If $\nu\in \widetilde H^{(s+1)/3.s+1}(B)$ for certain $s\geq 0$, then for any $j\leq s/3$ there exists
$\partial_t^j J_0(t,x,y;\nu)\in C_b(\mathbb R^t;\widetilde H^{s-3j}(\Sigma_-))$ and uniformly with respect to
$t\in\mathbb R$
\begin{equation}\label{2.26}
\|\partial_t^j J_0(t,\cdot,\cdot;\nu)\|_{\widetilde H^{s-3j}(\Sigma_-)} \leq c(b,s,L) \|\nu\|_{\widetilde H^{(s+1)/3.s+1}(B)}.
\end{equation}
If $\nu\in \widetilde H^{s/3.s}(B)$ for certain $s\geq 0$, then for any $j\leq s/3$ there exists
$\partial_t^j J_1(t,x,y;\nu)\in C_b(\mathbb R^t;\widetilde H^{s-3j}(\Sigma_-))$ and uniformly with respect to
$t\in\mathbb R$
\begin{equation}\label{2.27}
\|\partial_t^j J_1(t,\cdot,\cdot;\nu)\|_{\widetilde H^{s-3j}(\Sigma_-)} \leq c(b,s,L) \|\nu\|_{\widetilde H^{s/3.s}(B)}.
\end{equation}
\end{lemma}

\begin{proof}
The proof is based on the following inequality, established in \cite{F07-2}: let
$$
I(t,x) \equiv \int_{\mathbb R} e^{i\theta t} e^{r_j(\theta,a)x}w(\theta)\,d\theta,
$$
where $r_j(\theta,a)$, $j=1$ and $2$, are the roots of equation \eqref{2.18}, defined in \eqref{2.17}. Then there exists a positive constant $c$, such that uniformly with respect to $t\in\mathbb R$
\begin{equation}\label{2.28}
\|I(t,\cdot)\|_{L_2(\mathbb R_-)} \leq c\|(|\theta|^{1/3} +|a|^{1/2})w(\theta)\|_{L_2(\mathbb R)}.
\end{equation}

Now let
$$
J(t,x,y)\equiv  \sum\limits_{l=1}^{+\infty} \int_{\mathbb R}  e^{i\theta t} e^{r_j(\theta,b-\lambda_l)x}
w(\theta,l)\,d\theta \,\psi_l^{(m)}(y).
$$
Then it follows from \eqref{2.28} that uniformly with respect to $t\in\mathbb R$ since the system $\{\psi_l^{(m)}\}$ is also orthogonal in $L_2(0,L)$ and $\|\psi_l^{(m)}\|_{L_2(0,L)}\leq c(l/L)^m$
\begin{multline}\label{2.29}
\|J(t,\cdot,\cdot)\|_{L_2(\Sigma_-)} =
\Bigl(\sum\limits_{l=1}^{+\infty} 
\Bigl\|\int_{\mathbb R}  e^{i\theta t} e^{r_j(\theta,b-\lambda_l)x}
w(\theta,l)\,d\theta 
\Bigr\|_{L_2(\mathbb R_-^x)}^2\Bigl\|\psi^{(m)}_l\Bigr\|_{L_2(0,L)}^2\Bigr)^{1/2} \\ \leq
c(m,L)\Bigl(\sum\limits_{l=1}^{+\infty} \left\|\bigl(|\theta|^{1/3}+|b-\lambda_l|^{1/2}\bigr)w(\theta,l)\right\|_{L_2(\mathbb R^\theta)}^2 l^{2m}\Bigr)^{1/2}.
\end{multline}

Without loss of generality one can assume that $\nu \in\widetilde{\EuScript S}(\overline{B})$. Let $s$ be integer. Then for $3j+n+m=s$
\begin{equation}\label{2.30}
\partial_t^j \partial_x^n \partial_y^m J_0(t,x,y;\nu) = \sum\limits_{l=1}^{+\infty} \frac1{2\pi}
\int_{\mathbb R} (i\theta)^j \frac{r_1r_2^n e^{r_2x}-r_2r_1^ne^{r_1x}}{r_1-r_2}\widehat\nu(\theta,l)\,d\theta\, \psi_l^{(m)}(y)
\end{equation}
and inequalities \eqref{2.19}, \eqref{2.20} and \eqref{2.29} yield that
\begin{multline*}
\|\partial_t^j \partial_x^n \partial_y^m J_0(t,\cdot,\cdot;\nu)\|_{L_2(\Sigma_-)} \leq
c\Bigl(\sum\limits_{l=1}^{+\infty} \left\|\bigl(|\theta|^{2/3}+l^2\bigr)^{(3j+n+m+1)/2}\widehat\nu(\theta,l)\right\|_{L_2(\mathbb R^\theta)}^2\Bigr)^{1/2}  \\ =
c\|\nu\|_{\widetilde H^{(s+1)/3.s+1}(B)}.
\end{multline*}
Similarly,
\begin{equation}\label{2.31}
\partial_t^j \partial_x^n \partial_y^m J_1(t,x,y;\nu) = \sum\limits_{l=1}^{+\infty} \frac1{2\pi}
\int_{\mathbb R} (i\theta)^j  \frac{r_1^n e^{r_1x}-r_2^ne^{r_2x}}{r_1-r_2}\widehat\nu(\theta,l)\,d\theta\, \psi_l^{(m)}(y)
\end{equation}
and, therefore,
\begin{multline*}
\|\partial_t^j \partial_x^n \partial_y^m J_1(t,\cdot,\cdot;\nu)\|_{L_2(\Sigma_-)} \leq
c\Bigl(\sum\limits_{l=1}^{+\infty} \left\|\bigl(|\theta|^{2/3}+l^2\bigr)^{(3j+n+m)/2}\widehat\nu(\theta,l)\right\|_{L_2(\mathbb R^\theta)}^2\Bigr)^{1/2}  \\ =
c\|\nu\|_{\widetilde H^{s/3.s}(B)}.
\end{multline*}
Finally, use interpolation.
\end{proof}

\begin{lemma}\label{L2.7}
Let $\nu\in \widetilde H^{s/3,s}(B)$, then for any $T>0$
\begin{equation}\label{2.32}
\|J_1(\cdot,\cdot,\cdot;\nu)\|_{ C_b(\overline{\mathbb R}_-^x;\widetilde H^{(s+1)/3.s+1}(B_T))} \leq  c(T,b,s,L)\|\nu\|_{\widetilde H^{s/3.s}(B)}.
\end{equation}
\end{lemma}

\begin{proof}
Without loss of generality one can assume that $\nu \in\widetilde{\EuScript S}(\overline{B})$. There exists $l_0$ such that for $l>l_0$and all $\theta$ and there exists $\theta_0\geq 1$ such that for $|\theta|\geq \theta_0$ and all $l$
\begin{equation}\label{2.33}
|r_1(\theta,b-\lambda_l)-r_2(\theta,b-\lambda_l)|\geq c_0(|\theta|^{1/3}+l).
\end{equation}
Divide $\nu$ into two parts:
$$
\nu_0(t,y)\equiv \sum\limits_{l=1}^{l_0} \EuScript F^{-1}_t\bigr[ \widehat \nu(\theta,l)\eta(\theta_0+1-|\theta|)\bigr](x)\psi_l(y),\quad
\nu_{1}(t,y)\equiv \nu(t,y)-\nu_0(t,y).
$$
For $\nu_0$ inequality \eqref{2.27} yields, that for any $j$ and $m$
\begin{multline*}
\|\partial_t^j\partial_y^m J_1(\cdot,\cdot,\cdot;\nu_0)\|_{C_b(\overline{\mathbb R}_-^x;L_2(B_T))} \leq
T^{1/2}\sup\limits_{t\in [0,T]} \|\partial_t^j J_1(t,\cdot,\cdot;\nu_0)\|_{\widetilde H^{m+1}(\Sigma_-)} \\ \leq
c(T,b,j,m,L)\|\nu_0\|_{\widetilde H^{(3j+m+1)/3,3j+m+1}(B)} \leq c(T,b,s,j,m,L) \|\nu\|_{\widetilde H^{s/3.s}(B)}.
\end{multline*}
For $\nu_1$ by virtue of \eqref{2.33}
\begin{multline*}
\|J_1(\cdot,\cdot,\cdot;\nu_0)\|_{ C_b(\overline{\mathbb R}_-^x;\widetilde H^{(s+1)/3.s+1}(B))}  \\ \leq
c\Bigl(\sum\limits_{l=1}^{+\infty} \|\frac{(|\theta|^{2/3}+l^2)^{(s+1)/2}}{|\theta|^{1/3}+l}
\widehat\nu_1(\theta,l)\|^2_{L_2(\mathbb R^\theta)}\Bigr)^{1/2} \leq
c_1(b,s) \|\nu\|_{\widetilde H^{s/3.s}(B)}.
\end{multline*}
\end{proof}

\begin{lemma}\label{L2.8}
Let $\nu_0\in \widetilde H^{1/3,1}(B)$, $\nu_1\in L_2(B)$ and $\nu_0(t,y)=\nu_1(t,y)=0$ for $t<0$, then the function $u(t,x,y)\equiv J_0(t,x,y;\nu_0)+J_1(0,x,y;\nu_1)$ for any $T>0$ is a weak solution from the space$Y_0(\Pi_T^-)$ to problem \eqref{2.1} (for $f\equiv 0$), \eqref{1.2} (for $u_0\equiv 0$), \eqref{1.4}, \eqref{2.9}.
\end{lemma}

\begin{proof}
First let $\nu_0,\nu_1\in C_0^\infty(B_+)$. Consider the smooth solution $u(t,x,y)$ to the considered problem constructed in Lemma~\ref{L2.3}. For any $p=\varepsilon+i\theta$, 
where $\varepsilon>0$, define the Laplace--Fourier transform-coefficients
$$
\widetilde u(p,x,l) \equiv \int_{\mathbb R_+}\!\! \int_0^L e^{-pt} \psi_l(y) u(t,x,y)\,dydt.
$$
The function $\widetilde u(p,x,l)$ solves a problem
\begin{gather*}
p\widetilde u(p,x,l)+b\widetilde u_x(p,x,l)+\widetilde u_{xxx}(p,x,l)-\lambda_l \widetilde u_x(p,x,l)=0,\\
\widetilde u(p,0,l)=\widetilde \nu_0(p,l)\equiv \int_{\mathbb R_+}\!\! \int_0^L
e^{-pt}\psi_l(y)\nu_0(t,y)\,dydt, \quad
\widetilde u_x(p,0,l)=\widetilde \nu_1(p,l),
\end{gather*}
whence, since $\widetilde u(p,x,l)\to 0$ as $x\to -\infty$, it follows, that
$$
\widetilde u(p,x,l)=\frac {z_1e^{z_2x}-z_2e^{z_1x}}
{z_1-z_2} \widetilde \nu_0(p,l)
+ \frac {e^{z_1x}-e^{z_2x}} {z_1-z_2} \widetilde \nu_1(p,l).
$$
where $z_j=z_j(p,b-\lambda_l)$ are defined in \eqref{2.16} for $a=b-\lambda_l$.
Using the formula of inversion of the Laplace transform we find, that the Fourier coefficients of the function $u(t,x,\cdot)$ are the following:
$$
\widehat u(t,x,l) = e^{\varepsilon t} \EuScript F_t^{-1}\left[\frac {z_1e^{z_2x}-z_2e^{z_1x}}
{z_1-z_2}\widetilde \nu_0(\varepsilon+i\theta,l) + \frac {e^{z_1x}-e^{z_2x}} {z_1-z_2} \widetilde \nu_1(\varepsilon+i\theta,l)\right](t)
$$
and, therefore,
\begin{multline*}
u(t,x,y) \\= \sum\limits_{l=1}^{+\infty} e^{\varepsilon t} \EuScript F_t^{-1}\left[\frac {z_1e^{z_2x}-z_2e^{z_1x}}
{z_1-z_2}\widetilde \nu_0(\varepsilon+i\theta,l) +\frac {e^{z_1x}-e^{z_2x}} {z_1-z_2} \widetilde \nu_1(\varepsilon+i\theta,l)\right](t)\psi_l(y).
\end{multline*}
Passing to the limit as $\varepsilon\to+0$, we derive that $u(t,x,y) \equiv J_0(t,x,y;\nu_0)+J_1(0,x,y;\nu_1)$.

In the general case approximate the function $\mu$ by smooth ones, pass to the limit on the basis of estimates \eqref{2.26}, \eqref{2.27}, \eqref{2.25}, \eqref{2.32} for $s=0$, \eqref{2.23} for $s=1$ and use the uniqueness result. 
\end{proof}

\begin{lemma}\label{L2.9}
Let $u_0\in \widetilde H^k(\Sigma_-)$,
$\nu_0\in \widetilde H^{(k+1)/3,k+1}(B_T)$, $\nu_1\in \widetilde H^{k/3,k}(B_T)$, $f\in M_k(\Pi_T^-)$ for certain $T>0$, $k\geq 0$. Assume also that $\partial_t^j \nu_0(0,y) \equiv \widetilde \Phi_j(0,y)$ for $j<k/3$, $\partial_t^j \nu_1(0,y) \equiv \widetilde\Phi_{jx}(0,y)$ for $j<(k-1)/3$. Then there exists a unique
solution $u(t,x,y)\in Y_k(\Pi_T^-)$ to problem \eqref{2.1}, \eqref{1.2}, \eqref{1.4}, \eqref{2.9} and for any $t_0\in (0,T]$
\begin{multline}\label{2.34}
\|u\|_{Y_k(\Pi_{t_0}^-)}
\leq c(T,k,b,L)\Bigl(\|u_0\|_{\widetilde H^k(\Sigma_-)}+
\|\nu_0\|_{\widetilde H^{(k+1)/3,k+1}(B_T)}+
\|\nu_1\|_{\widetilde H^{k/3,k}(B_T)} \\+
t_0^{1/6}\|f\|_{M_k(\Pi_{t_0}^-)}
 +\sum_{j=0}^{j_0-1} \|\partial_t^j
f\big|_{t=0}\|_{\widetilde H^{k-3(j+1)}(\Sigma_-)}\Bigr), \quad j_0=[(k+1)/3].
\end{multline}
\end{lemma}

\begin{proof}
Extend the functions $u_0$ and $f$ to the whole real axis with respect to $x$ in the
classes $\widetilde H^k(\Sigma)$ and $M_k(\Pi_T)$ respectively and
consider the solution $U(t,x,y)$ to the initial value problem
\eqref{2.1}, \eqref{1.2}, \eqref{1.4} in the class $Y_k(\Pi_T)$ given by
Lemma~\ref{L2.1}. Note that 
$$\widetilde \nu_0\equiv \nu_0-U|_{x=0}\in \widetilde H^{(k+1)/3,k+1}(B_T),\quad
\widetilde \nu_1\equiv \nu_1-U_x|_{x=0}\in \widetilde H^{k/3,k}(B_T),
$$
and by virtue of the compatibility conditions
$\partial_t^j\widetilde \nu_0\big|_{t=0}=0$ for $j<k/3$,  
$\partial_t^j\widetilde \nu_1\big|_{t=0}=0$ for $j<(k-1)/3$, 
so the functions $\widetilde\nu_0$, $\widetilde\nu_1$ can be extended in the same spaces to the whole strip $B$, such that $\widetilde\nu_0(t,y)=\widetilde\nu_1(t,y)=0$ for $t<0$.
Then Lemmas~\ref{L2.1}--\ref{L2.8} for the function
$$
u(t,x,y)\equiv U(t,x,y)+J_0(t,x,y;\widetilde\nu_0)+J_1(t,x,y;\widetilde\nu_1)
$$
provide the desired result.
\end{proof}

Now consider the problem in $Q_T$.

\begin{lemma}\label{L2.10}
Let $u_0\in \widetilde H^k$,
$\mu_0,\nu_0\in \widetilde H^{(k+1)/3,k+1}(B_T)$, $\nu_1\in \widetilde H^{k/3,k}(B_T)$, $f\in M_k(Q_T)$ for certain $T>0$, $k\geq 0$. Assume also that $\partial_t^j \mu_0(0,y) \equiv \widetilde \Phi_j(0,y)$, $\partial_t^j \nu_0(0,y) \equiv \widetilde \Phi_j(R,y)$ for $j<k/3$, $\partial_t^j \nu_1(R,y) \equiv \widetilde\Phi_{jx}(R,y)$ for $j<(k-1)/3$. Then there exists a unique
solution $u(t,x,y)\in Y_k(Q_T)$ to problem \eqref{2.1}, \eqref{1.2}--\eqref{1.4} and for any $t_0\in (0,T]$
\begin{multline}\label{2.35}
\|u\|_{Y_k(Q_{t_0})} 
\leq c(T,k,b,R,L)\Bigl(\|u_0\|_{\widetilde H^k}+
\|\mu_0\|_{\widetilde H^{(k+1)/3,k+1}(B_T)}+
\|\nu_0\|_{\widetilde H^{(k+1)/3,k+1}(B_T)} \\+
\|\nu_1\|_{\widetilde H^{k/3,k}(B_T)} +
t_0^{1/6}\|f\|_{M_k(Q_{t_0})}
 +\sum_{j=0}^{j_0-1} \|\partial_t^j
f\big|_{t=0}\|_{\widetilde H^{k-3(j+1)}}\Bigr).
\end{multline}
\end{lemma}

\begin{proof}
Solutions to the considered problem (similarly to the corresponding problem
in \cite{F08}) are constructed in the form
\begin{equation}\label{2.36}
u(t,x,y)=w(t,x,y)+v(t,x,y),
\end{equation}
where $w(t,x,y)$ is a solution to an initial-boundary value
problem in $\widetilde\Pi^-_{T}=(0,T)\times \widetilde \Sigma_{-}$, $\widetilde \Sigma_- = (-\infty,R)\times (0,L)$  for equation \eqref{2.1} with initial and boundary conditions \eqref{1.2} for $(x,y)\in \widetilde\Sigma_-$, \eqref{1.4} for $(t,x)\in (0,T)\times (-\infty,R)$ and \eqref{2.9} (where $x=0$ is substituted by $x=R$) in the class $Y_k(\widetilde \Pi^-_{T})$. Then according to \eqref{2.34} ($u_0$ and $f$ are extended to $x<0$ in a appropriate way)
\begin{multline}\label{2.37}
\|w\|_{Y_k(\widetilde \Pi_{t_0}^-)}
\leq c(T,k,b,L)\Bigl(\|u_0\|_{\widetilde H^k}+
\|\nu_0\|_{\widetilde H^{(k+1)/3,k+1}(B_T)}+
\|\nu_1\|_{\widetilde H^{k/3,k}(B_T)} \\+
t_0^{1/6}\|f\|_{M_k(\widetilde Q_{t_0})}
 +\sum_{j=0}^{j_0-1} \|\partial_t^j
f\big|_{t=0}\|_{\widetilde H^{k-3(j+1)}}\Bigr).
\end{multline}
Moreover, 
$$
\widetilde\mu_0(t,y)\equiv \mu_0(t,y)-w(t,0,y)\in
H^{(k+1)/3,k+1}(B_T),
$$
by virtue of the compatibility conditions on the line $(0,0,y)$
$\partial_t^j\widetilde \mu_0(0,y)\equiv 0$ for $j<k/3$ and
\begin{multline}\label{2.38}
\|\widetilde\mu_0\|_{\widetilde H^{(k+1)/3,k+1}(B_T)}
\leq c(T,k,b)\Bigl(\|u_0\|_{\widetilde H^k}+
\|\mu_0\|_{\widetilde H^{(k+1)/3,k+1}(B_T)}\\+
\|\nu_0\|_{\widetilde H^{(k+1)/3,k+1}(B_T)} +
\|\nu_1\|_{\widetilde H^{k/3,k}(B_T)} +
t_0^{1/6}\|f\|_{M_k(Q_{t_0})}
 +\sum_{j=0}^{j_0-1} \|\partial_t^j
f\big|_{t=0}\|_{\widetilde H^{k-3(j+1)}}\Bigr).
\end{multline}
In particular, the function $\widetilde\mu_0$ can be considered as extended in the same class to the whole strip $B$ such that $\widetilde\mu_0(t,y) = 0$ for $t<0$.
 
Consider in $Q_T$ a problem for the function $v$:
\begin{gather}\label{2.39}
v_t+bv_x+v_{xxx}+v_{xyy}=0, \\
\label{2.40} 
v\big|_{t=0}=0, \quad v\big|_{x=0}=\widetilde\mu_0, \quad
v\big|_{x=R}=v_x\big|_{x=R}=0
\end{gather}
also with corresponding boundary conditions \eqref{1.4}.
In order to construct a solution to this problem we consider for $x\geq 0$ the
boundary potential $J(t,x,y;\mu)$ for an arbitrary function
$\mu\in \widetilde H^{(k+1)/3,k+1}(B)$, $\mu(t,y)=0$ for $t<0$. 
Such a potential was introduced in \cite{F17} as a solution to an initial-boundary value problem in $\Pi_T^+ =(0,T)\times \Sigma_+$, $\Sigma_+=\mathbb R_+ \times (0,L)$, for equation \eqref{2.1} in the case $f\equiv 0$ with zero initial condition \eqref{1.2} for $(x,y)\in \Sigma_+$, boundary condition \eqref{1.4} for $(t,x)\in (0,T)\times\mathbb R_+$ and boundary condition
\begin{equation}\label{2.41}
u(t,0,y)=\mu(t,y),\quad (t,y)\in B_T.
\end{equation} 
According to \cite{F17} the function $J$ is infinitely differentiable for $x>0$ and 
for any $\delta\in (0,T]$
\begin{multline}\label{2.42}
\|J(\cdot,R,\cdot;\mu)\|_{\widetilde H^{(k+1)/3,k+1}(B_\delta)}+
\|\partial_x J(\cdot,R,\cdot;\mu)\|_{\widetilde H^{k/3,k}(B_\delta)} \\ \leq
c(T,k,b,R,L)\delta^{1/2}\|\mu\|_{L_2(B_\delta)}.
\end{multline}
Moreover, $\partial_t^j J(0,R,y;\mu)=\partial_t^j J_x(0,R,y;\mu)\equiv 0$ for all $j$.

Consider in the domain $\widetilde \Pi_{\delta}^-$ the problem of 
\eqref{2.1}, \eqref{1.2}, \eqref{1.4}, \eqref{2.9} (for $x=R$) type, where
$u_0\equiv 0$, $f\equiv 0$, $\nu_0\equiv -J(\cdot,R,\cdot;\mu)$,
$\nu_1\equiv -\partial_x J(\cdot,R,\cdot;\mu)$. A solution to this
problem $V\in Y_k(\widetilde \Pi_\delta^-)$ exists and, in particular,
\begin{equation}\label{2.43}
\begin{aligned}
&\|V(\cdot,0,\cdot)\|_{\widetilde H^{(k+1)/3,k+1}(B_\delta)}  \\
&\leq c(T,k,b,L) \left(\|J(\cdot,R,\cdot;\mu)\|_{\widetilde H^{(k+1)/3,k+1}(B_\delta)}+
\|\partial_x J(\cdot,R,\cdot;\mu)\|_{\widetilde H^{k/3,k}(B_\delta)}\right).
\end{aligned}
\end{equation}
Moreover, it is obvious that $\partial_t^j V(0,0,y) \equiv 0$ if $j<k/3$.

Consider a linear operator $\Gamma: \mu \mapsto V(\cdot,0,\cdot)$
in the space $\widetilde H^{(k+1)/3,k+1}(B_\delta)$, $\partial_t^j\mu(0,y)\equiv 0$ if $j<k/3$. For small $\delta=\delta(T,k,b,L)$ estimates \eqref{2.42} and
\eqref{2.43} provide that the operator $(E+\Gamma)$ is invertible
($E$ is the identity operator) and setting $\mu\equiv
(E+\Gamma)^{-1}\widetilde\mu_0$ we obtain the desired solution to problem
\eqref{2.39}, \eqref{2.40}, \eqref{1.4}
$$
v(t,x,y)\equiv J(t,x,y;\mu) + V(t,x,y),
$$
where (also with the use of the corresponding estimate on $J$ from \cite{F17})
\begin{equation}\label{2.44}
\|v\|_{Y_k(Q_\delta)} \leq c(T,k,b,L)\|\widetilde\mu_0\|_{\widetilde H^{(k+1)/3,k+1}(B_T)}.
\end{equation}

Thus, the solution $u(t,x,y)$ to problem \eqref{2.1},
\eqref{1.2}--\eqref{1.4} in the domain $Q_\delta$ is constructed
and according to (\ref{2.36})--(\ref{2.38}) and (\ref{2.44})  is
evaluated in the space $Y_k(Q_\delta)$ by the right part of
(\ref{2.35}). Moving step by step ($\delta$ is constant) we obtain
the desired solution in the whole domain $Q_T$.

Uniqueness of weak solutions to problem \eqref{2.1},
\eqref{1.2}--\eqref{1.4} in $L_2(Q_T)$ succeeds from existence of
smooth solutions to the adjoint problem
\begin{gather*}
\phi_t+b\phi_x+\phi_{xxx}+\phi_{xyy}=f\in C_0^\infty(Q_T), \\
\phi\big|_{t=T}=0, \quad
\phi\big|_{x=0}=\phi_x\big|_{x=0}=\phi\big|_{x=R}=0
\end{gather*}
and with the corresponding boundary conditions of \eqref{1.4} type,
which after simple change of variables transforms to the original one.
\end{proof}

\begin{remark}\label{R2.2}
In further lemmas of this section all intermediate argument is performed for smooth solutions constructed in Lemma~\ref{L2.10} with consequent pass to the limit on the basis of obtained estimates due to linearity of the problem. 
\end{remark}

\begin{lemma}\label{L2.11}
Let $u_0\in L_2$, $\mu_0=\nu_0\equiv 0$, $\nu_1\in L_2(B_T)$, $f\equiv f_0 +f_{1x}$, where $f_0\in L_1(0,T;L_2)$, $f_1\in L_2(Q_T)$. Then there exist a (unique) weak solution to problem \eqref{2.1}, \eqref{1.2}--\eqref{1.4} from the space $X(Q_T)$ and a function $\mu_1\in L_2(B_T)$, such that for any function $\phi\in L_2(0,T;\widetilde H^2)$, $\phi_t, \phi_{xxx}, \phi_{xyy}\in L_2(Q_T)$, $\phi\big|_{t=T}=0$, $\phi\big|_{x=0}=\phi\big|_{x=R} =0$, the following equality holds:
\begin{multline}\label{2.45}
\iiint_{Q_T}\Bigl[u(\phi_t+b\phi_x+\phi_{xxx}+\phi_{xyy}) +f_0\varphi -f_1\phi_x \Bigr]\,dxdydt \\
+\iint u_0\phi\big|_{t=0}\,dxdy +
\iint_{B_T} \Bigl[\nu_1\phi_x\big|_{x=R}-\mu_1\phi_x\big|_{x=0}\Bigr]\,dydt =0.
\end{multline}
Moreover, for $t\in (0,T]$
\begin{multline}\label{2.46}
\|u\|_{X(Q_t)} +\|\mu_1\|_{L_2(B_t)}\leq c(T,b,R)\Bigl(\|u_0\|_{L_2} + \|\nu_1\|_{B_T} 
+\|f_0\|_{L_1(0,t;L_2)} \\ +\|f_1\|_{L_2(Q_{t})}\Bigr),
\end{multline}
and if either $\rho(x)\equiv 1$ or $\rho(x)\equiv 1+x$
\begin{multline}\label{2.47}
\iint u^2(t,x,y)\rho(x)\,dxdy + \int_0^t \!\! \iint (3u_x^2 +u_y^2 -bu^2)\rho'(x)\,dxdyd\tau + \rho(0)\iint_{B_t} \mu_1^2\,dyd\tau \\ =
\iint u_0^2\rho(x)\,dxdy + \rho(R)\iint_{B_t} \nu_1^2\,dyd\tau +
2\int_0^t \!\! \iint f_0 u\rho(x)\,dxdyd\tau  \\
-2\int_0^t \!\! \iint f_1\bigl(u\rho(x)\bigr)_x\,dxdyd\tau.
\end{multline}
\end{lemma}

\begin{proof}
Multiplying \eqref{2.1} by $2u(t,x,y)\rho(x)$ and integrating over $\Omega$, we find that
\begin{multline}\label{2.48}
\frac d{dt} \iint u^2\rho\,dxdy +\rho(0)\int_0^L u_{x}^2\big|_{x=0}\,dy +
\iint (3u_{x}+u_{y}^2 -bu^2)\rho' \,dxdy \\ = \rho(R)\int_0^L \nu_1^2\,dy +
2\iint f_0 u\rho\,dxdy -2\iint f_1(u\rho)_x\,dxdy.
\end{multline}
Note that
\begin{multline}\label{2.49}
\Bigl|\iint f_1(u\rho)_x\,dxdy\Bigr| \leq c\|f_1\|_{L_2}
\bigl\|(|u_{x}|+|u|)\bigr\|_{L_2}\\ \leq \varepsilon
\iint \bigl(u_{x}^2+u^2\bigr)\,dxdy +
c(\varepsilon)\|f_1\|_{L_2}^2,
\end{multline}
where $\varepsilon>0$ can be chosen arbitrarily small.
Equality \eqref{2.48}  for $\rho\equiv 1+x$ and inequality \eqref{2.49} imply that that for smooth solutions
\begin{equation}\label{2.50}
\|u\|_{X(Q_T)} + \|u_{x}\big|_{x=0}\|_{L_2(B_T)} \leq c.
\end{equation}
The end of the proof is standard.
\end{proof}

\begin{remark}\label{R2.3}
The method of construction of weak solution in Lemma~\ref{L2.13} via closure ensures that $u\big|_{x=0}=u\big|_{x=R}=0$ in the trace sense (this fact can be also easily derived from equality \eqref{2.45}, since $u_x \in L_2(Q_T)$). Moreover, if $f\in L_2(Q_T)$ then according to Lemma~\ref{2.12} $u\in Y_0(Q_T)$ and, in particular, $\mu_1\equiv u_x\big|_{x=0}$.
\end{remark}

\begin{lemma}\label{L2.12}
Let $u_0\in \widetilde H^{(0,1)}$, $\mu_0=\nu_0=\nu_1\equiv 0$, $f\in L_2(Q_T)$. Then for the unique weak solution $u(t,x,y)\in X(Q_T)$ to problem \eqref{2.1}, \eqref{1.2}--\eqref{1.4} $u_y\in C([0,T];L_2)$, $|Du_y|\in L_2(Q_T)$ and for any $t\in (0,T]$
\begin{multline}\label{2.51}
\iint u_y^2(t,x,y)\,dxdy + \int_0^t\!\! \iint |Du_y|^2\,dxdyd\tau  \\ \leq 
(1+R)\iint u_{0y}^2\,dxdy + b\int_0^t\!\! \iint u_y^2 \, dxdyd\tau -
2\int_0^t\!\! \iint (1+x) f u_{yy}\,dxdyd\tau.
\end{multline}
\end{lemma}

\begin{proof}
Multiply \eqref{2.1} by $-2(1+x)u_{yy}(t,x,y)$ and integrate over $\Omega$, then
\begin{multline*}
\frac{d}{dt} \iint (1+x)u_y^2\,dxdy + \int_0^L u_{xy}^2\big|_{x=0}\,dy + 
\iint (3u_{xy}^2+u_{yy}^2- bu_y^2)\,dxdy \\ =
-2 \iint (1+x)fu_{yy}\,dxdy,
\end{multline*}
whence the assertion of the lemma obviously follows.
\end{proof}

\begin{lemma}\label{L2.13}
Let $u_0\in \widetilde H^2$, $u_0\big|_{x=0}=u_0\big|_{x=R}=u_{0x}\big|_{x=R}\equiv 0$ and $u_{0xxx},u_{0xyy} \in L_2$, $\mu_0=\nu_0=\nu_1\equiv 0$, $f\in C([0,T];L_2)$, $f_t\in L_2(0,T;H^{-1})$. Then for the (unique) weak solution to problem \eqref{2.1}, \eqref{1.2}--\eqref{1.4} from the space $X(Q_T)$  there exists $u_t\in X(Q_T)$, which is the weak solution to problem of \eqref{2.1}, \eqref{1.2}--\eqref{1.4} type, where $f$ is substituted by $f_t$, $u_0$ -- by $\bigl(f\big|_{t=0}-bu_{0x} -u_{0xxx}- u_{0xyy}\bigr)$, $\mu_0=\nu_0=\nu_1\equiv 0$.
\end{lemma}

\begin{proof}
The proof for the function $v\equiv u_t$ is similar to Lemma~\ref{L2.11}.
\end{proof}

\begin{lemma}\label{L2.14}
Let the hypothesis of Lemma~\ref{L2.13} be satisfied and, in addition, $f\in L_1(0,T;\widetilde H^{(0,2)})$. Then there exists a (unique) solution to problem \eqref{2.1}, \eqref{1.2}--\eqref{1.4} from the space $X^{2}(Q_T)$ and for any $t\in [0,T]$
\begin{multline}\label{2.52}
\|u\|^2_{X^2(Q_t)} \leq c(T,b,R)\Bigl(\|u_{0yy}\|^2_{L_2} +
\|f\|^2_{C([0,t];L_2)}+ \|u\|^2_{C([0,t];L_2)} + \|u_t\|^2_{C([0,t];L_2)} \\+
\sup\limits_{\tau\in (0,t]}\Bigl|\int_0^{\tau} \!\!\iint (1+x)f_{yy}u_{yy}\,dxdyds\Bigr| \Bigr).
\end{multline}
\end{lemma}

\begin{proof}
For smooth solutions differentiating equality \eqref{2.1} twice with respect to $y$, multiplying the obtained equality by $2u_{yy}(t,x,y)\rho(x)$, $\rho(x)\equiv (1+x)$, and integrating over $\Omega$ we derive, that
\begin{multline}\label{2.53}
\frac d{dt} \iint u_{yy}^2\rho\,dxdy +\int_0^L u_{xyy}^2\big|_{x=0}\,dy +
\iint (3u_{xyy}^2+u_{yyy}^2-bu^2_{yy})\,dxdy \\ =  
2\iint f_{yy}u_{yy}\rho\,dxdy,
\end{multline}
whence obviously follows that
\begin{equation}\label{2.54}
\|u_{yy}\|_{X(Q_T)} \leq c.
\end{equation}
Hence, for the weak solution also $u_{yy}\in X(Q_T)$. Lemmas~\ref{L2.11} and~\ref{L2.13} provide, that $u,u_t \in X(Q_T)$. Write equality \eqref{2.1} in the form
\begin{equation}\label{2.55}
u_{xxx} = f-u_t -bu_x -u_{xyy}.
\end{equation}
Then, inequality \eqref{1.22} for $j=2$ and \eqref{2.55} yield that 
\begin{multline}\label{2.56}
\|u_{xx}\|_{L_2} \leq c(R)\bigl(\|u_{xxx}\|_{H^{(-1,0)}} + 
\|u\|_{L_2}\bigr) \\ \leq
c(b,R)\bigl(\|f\|_{L_2} + 
\|u_t\|_{L_2} + \|u_{yy}\|_{L_2} + \|u\|_{L_2}\bigr).
\end{multline}
Since
$$
\iint u_{xy}^2\,dxdy = \iint u_{xx}u_{yy}\,dxdy,
$$
estimates \eqref{2.54} and \eqref{2.56} yield that $u\in C([0,T];\widetilde H^{2})$ and
\begin{equation}\label{2.57}
\|u(t,\cdot,\cdot)\|_{\widetilde H^{2}} \leq c\bigl( \|f\|_{L_2} + 
\|u_t\|_{L_2} + \|u_{yy}\|_{L_2} + \|u\|_{L_2}\bigr).
\end{equation}
Next, 
$$
\iint u_{xxy}^2\,dxdy = \iint u_{xxx}u_{xyy}\,dxdy 
+\int_0^L (u_{xyy}u_{xx})\big|_{x=0}\,dy
$$
and inequality \eqref{1.19} provides that
\begin{equation}\label{2.58}
\iint u_{xxy}^2\,dxdy \leq \iint (u_{xxx}^2+u_{xyy}^2)\,dxdy +
 \int_0^L u^2_{xyy}\big|_{x=0}\,dy +
c\iint u_{xx}^2\,dxdy.
\end{equation}
From equality \eqref{2.55} we derive, that
\begin{equation}\label{2.59}
\iint u_{xxx}^2\,dxdy \leq c\iint (f^2+u_t^2+b^2u_x^2+u_{xyy}^2)\,dxdy,
\end{equation}
and combining \eqref{2.53}, \eqref{2.57}--\eqref{2.59} finish the proof.
\end{proof}

\begin{lemma}\label{L2.15}
Let the hypothesis of Lemma~\ref{L2.13} be satisfied and, in addition, $u_0\in \widetilde H^3$, $f\in L_2(0,T;\widetilde H^{(0,2)})$. Then there exists a (unique) solution to problem \eqref{2.1}, \eqref{1.2}--\eqref{1.4} from the space $X^{3}(Q_T)$ and for any $t\in (0,T]$
\begin{multline}\label{2.60}
\|u\|_{X^{3}(Q_t)} \leq c(T,b,R,L)\bigl(\|u_0\|_{\widetilde H^{3}} +
\|f\|_{C([0,t];L_2)} + \|f\|_{L_2(0,t;\widetilde H^{(0,2)})} \\ 
+\|f_{t}\|_{L_2(0,t;H^{-1})}\bigr).
\end{multline} 
\end{lemma}

\begin{proof}
First of all note that hypotheses of Lemmas~\ref{L2.11} (for $f_1\equiv 0$), \ref{L2.13} and~\ref{L2.14} are satisfied. Therefore, taking into account also Remark~\ref{R2.3} we derive for smooth solutions that
\begin{equation}\label{2.61}
\|u\|_{X^2(Q_T)} + \|u_x\big|_{x=0}\|_{L_2(B_T)} + 
\|u_t\|_{X(Q_T)} + \|u_{tx}\big|_{x=0}\|_{L_2(B_T)}  \leq c.
\end{equation}
Next, differentiating equality \eqref{2.1} twice with respect to $y$, multiplying the obtained equality by $-2u_{yyyy}(t,x,y)\rho(x)$, $\rho(x)\equiv (1+x)$ and integrating over $\Omega$ we derive similarly to  \eqref{2.53} that
\begin{multline}\label{2.62}
\frac d{dt} \iint u_{yyy}^2\rho\,dxdy +\int_0^L u_{xyyy}^2\big|_{x=0}\,dy +
\iint (3u_{xyyy}+u_{yyyy}^2-bu^2_{yyy})\,dxdy \\ = - 2\iint f_{yy}u_{yyyy}\rho\,dxdy.
\end{multline}
Here
$$
\Bigl|2\iint f_{yy}u_{yyyy}\rho\,dxdy \Bigr| \leq \varepsilon \iint u_{yyyy}^2\,dxdy +
\frac {(1+R)^2}\varepsilon\iint f_{yy}^2 \,dxdy,
$$
where $\varepsilon>0$ can be chosen arbitrarily small, and equality \eqref{2.62} yields that
\begin{equation}\label{2.63}
\|u_{yyy}\|_{X(Q_T)} + \|u_{xyyy}\big|_{x=0}\|_{L_2(B_T)}  \leq c.
\end{equation}
Again apply equality \eqref{2.55}. Then it follows from \eqref{2.63} that we have the suitable estimate on $u_{xxxy}$ in the space $L_2(Q_T)$. Similarly to \eqref{2.58}
$$
\iint u_{xxyy}^2\,dxdy \leq \iint (u_{xxxy}^2+u_{xyyy}^2)\,dxdy +
 \int_0^L u^2_{xyyy}\big|_{x=0}\,dy +
c\iint u_{xxy}^2\,dxdy,
$$
whence follows the suitable estimate on $u_{xxyy}$ in $L_2(Q_T)$ and, as a result, on $u_y$ in $L_2(0,T;\widetilde H^3)$. One more application of \eqref{2.55} yields the estimate on $u_{xxxx}$ in $L_2(Q_T)$. Therefore,
\begin{equation}\label{2.64}
\|u\|_{L_2(0,T;\widetilde H^4)} \leq c.
\end{equation}

Consider the extensions of the functions $u$ and $f$ for $y\in (L,2L]$ and $y\in [-L,0)$ in the case a) by the  even reflections through $y=L$ and $y=0$, in the case b) -- by the odd ones, in the case c) -- by the corresponding combination of these methods, in the case d) -- by the periodic extension. Then the functions $u$ and $f$ remain smooth in the more wide domain $[0,T]\times [0,R] \times [-L,2L]$, and equality \eqref{2.1} also remains valid. Let $\eta_L(y) \equiv \eta(1+y/L)\eta(2-y/L)$, $\widetilde u(t,x,y) \equiv u(t,x,y)\eta_L(y)$, $\widetilde f(t,x,y) \equiv f(t,x,y)\eta_L(y)$. 
Now we apply the inequality (see, e.g. \cite{LM}) for the domain $\widetilde\Omega=(0,R)\times \mathbb R^y$
$$
\|g\|_{H^2(\widetilde\Omega)}\leq c\bigl(\|\Delta g\|_{L_2(\widetilde\Omega)}+
\|g\big|_{\partial\widetilde\Omega}\|_{H^{3/2}(\mathbb R)}+
\|g\|_{H^1(\widetilde\Omega)}\bigr)
$$
for the function $g\equiv \widetilde u_x$. Note that $g\big|_{x=R}=0$ and 
$$
\Delta_{x,y} \widetilde u_x = \widetilde f -\widetilde u_t -b\widetilde u_x +2u_{xy}\eta_L'+u_x\eta_L''.
$$
It follows from (\ref{2.61}) that
$$
\|\Delta_{x,y} \widetilde u_x\|_{C([0,T];L_2(\widetilde\Omega))} \leq c.
$$
Moreover, by virtue of (\ref{2.61}), (\ref{2.63}) and embedding
$H^2(\widetilde\Omega)\subset H^{3/2}(\{{x=0}\}\times \mathbb R^y)$ (see \cite{LM})
\begin{multline*}
\|u_x\big|_{x=0}\|_{C([0,T];H^{3/2}(\mathbb R))} \leq 
\|u_{0x}\big|_{x=0}\|_{H^{3/2}(\mathbb R)}\\
+2\|u_{tx}\big|_{x=0}\|_{L_2((0,T)\times \mathbb R)}^{1/2}
\|u_x\big|_{x=0}\|_{L_2(0,T;H^3(\mathbb R))}^{1/2} \leq c.
\end{multline*}
Therefore,
\begin{equation}\label{2.65}
\|u_x\|_{C([0,T];H^2)} \leq c.
\end{equation}
Estimates \eqref{2.61}, \eqref{2.63}--\eqref{2.65} provide the desired result.
\end{proof}

At the end of this section consider the particular case of problem \eqref{2.1}, \eqref{1.2}--\eqref{1.4} in $Q_T$ for $\mu_0=\nu_0=\nu_1\equiv 0$, $f\equiv 0$. Denote its solution by $Pu_0$, then it succeeds from Lemma~\ref{L2.10} that the operator $P$ is linear and bounded from $L_2$ to $Y_0(Q_T)$. Moreover, it easily follows from \eqref{2.47} that
\begin{equation}\label{2.66}
\bigl\|\partial_x(Pu_0)\big|_{x=0}\bigr\|_{L_2(B_T)} \leq \|u_0\|_{L_2}.
\end{equation}
For the controllability purposes we need the following observability result.

\begin{lemma}\label{L2.16} 
If condition \eqref{1.13} holds, then there exists a constant $c=c(T,b,R,L)>0$, such that
\begin{equation}\label{2.67}
\|u_0\|_{L_2} \leq c \bigl\|\partial_x(Pu_0)\big|_{x=0}\bigr\|_{L_2(B_T)}.
\end{equation}
\end{lemma}

\begin{proof}
In the smooth case multiplying \eqref{2.1} by $2(T-t)u(t,x,y)$ and integrating over $Q_T$ we find, that
$$
\iiint_{Q_T} u^2\,dxdydt -T\iint u_0^2\,dxdy +\iint_{B_T} (T-t)u_x^2\big|_{x=0}\,dydt =0,
$$
whence follows, that
\begin{equation}\label{2.68}
\iint u_0^2\,dxdy \leq \frac 1T \iiint_{Q_T} (Pu_0)^2\,dxdydt +
\iint_{B_T} \bigl(\partial_x(Pu_0)\big|_{x=0}\bigr)^2\,dydt.
\end{equation}
By continuity this estimate can be extended to any $u_0\in L_2$.

Now assume, that inequality \eqref{2.67} is not true. Then there exists a sequence $\{u_{0n}\in L_2\}_{n\in\mathbb N}$ such that
\begin{equation}\label{2.69}
\|u_{0n}\|_{L_2}=1\quad \forall\ n,\qquad \lim\limits_{n\to+\infty} 
\bigl\|\partial_x(Pu_{0n})\big|_{x=0}\bigr\|_{L_2(B_T)}=0.
\end{equation}
It follows from \eqref{2.47} that the sequence $\{Pu_{0n}\}$ is bounded in $L_2(0,T;H^1)$. Moreover, equality \eqref{2.1} provides that the sequence $\{\partial_t Pu_{0n}\}$ is bounded in $L_1(0,T;H^{-2})$ and the standard argument provides that $\{Pu_{0n}\}$ is precompact in $L_2(Q_T)$. Extract the subsequence $n'$, such that $\{Pu_{0n'}\}$ converges in $L_2(Q_T)$.  
It follows from \eqref{2.68}, \eqref{2.69} that $\{u_{0n'}\}$ converges in $L_2$ to a certain function $\widetilde u_0\in L_2$. Continuity of the operator $P$ and the second property \eqref{2.69} yield, that $P\widetilde u_0\in Y_0(Q_T)$ verifies $\partial_x(P\widetilde u_{0})\big|_{x=0}=0$. In particular, according to \eqref{2.45} for any function $\phi\in L_2(0,T;\widetilde H^2)$, $\phi_t, \phi_{xxx}, \phi_{xyy}\in L_2(Q_T)$, $\phi\big|_{t=T}=0$, $\phi\big|_{x=0}=\phi\big|_{x=R} =0$, the following equality holds:
\begin{equation}\label{2.70}
\iiint_{Q_T}P\widetilde u_0(\phi_t+b\phi_x+\phi_{xxx}+\phi_{xyy})\,dxdydt 
+\iint \widetilde u_0\phi\big|_{t=0}\,dxdy=0.
\end{equation}
For any natural $l$ let 
\begin{equation}\label{2.71}
v_l(t,x)\equiv \int_0^L (P\widetilde u_0)(t,x,y)\psi_l(y)\,dy,\quad
v_{0l}(x) \equiv \int_0^L \widetilde u_0(x,y)\psi_l(y)\,dy.
\end{equation}
Let $\vartheta(t,x)$ be an arbitrary function, such that $\vartheta\in L_2(0,T; H^3(0,R)\cap H_0^1(0,R))$, $\vartheta_t\in L_2((0,T)\times (0,R))$, $\vartheta\big|_{t=T}\equiv 0$. Choose $\phi(t,x,y)\equiv \vartheta(t,x)\psi_l(y)$, then it follows from \eqref{2.70}, \eqref{2.71}, that
\begin{equation}\label{2.72}
\iint_{(0,T)\times (0,R)} v_l\bigl(\vartheta_t+ (b-\lambda_l)\vartheta_x +\vartheta_{xxx}\bigr)\,dxdt +
\int_0^R v_{0l}\vartheta\big|_{t=0}\,dx =0.
\end{equation}
It means, that the function $v_l \in C([0,T];L_2(0,R))$, $v_{lx} \in C([0,R];L_2(0,T))$ is a weak solution in the rectangle $(0,T)\times (0,R)$ to an initial-boundary value problem
\begin{gather}\label{2.73}
v_t +(b-\lambda_l)v_x +v_{xxx}=0,\\ 
\label{2.74}
v\big|_{t=0}= v_{0l},\quad v\big|_{x=0}=v_x\big|_{x=0}=v\big|_{x=R}=v_x\big|_{x=R}=0.
\end{gather}
But the obvious generalization of results from \cite{R} (in that paper the case of the equation $v_t+v_x+v_{xxx}=0$ was considered) shows that under condition \eqref{1.13} (if $b-\lambda_l\leq 0$ there are no restrictions on $R$) $v_{0l}\equiv 0$ and, therefore, $\widetilde u_0\equiv 0$, which contradicts the fact, that $\|\widetilde u_0\|_{L_2}=1$.  
\end{proof}

\section{Existence of solutions}\label{S3}

Consider an auxiliary equation
\begin{equation}\label{3.1}
u_t+bu_x+u_{xxx}+u_{xyy}+(g(u))_x+(\psi(t,x,y)u)_x=f(t,x,y).
\end{equation}
The notion of a weak solution to problem \eqref{3.1}, \eqref{1.2}--\eqref{1.4} is similar to Definition~\ref{D1.1}.

\begin{lemma}\label{L3.1}
Let $g\in C^1(\mathbb R)$, $g(0)=0$, $|g'(u)|\leq c\ \forall u\in\mathbb R$, $\psi\in L_2(0,T;L_\infty)$, $u_0\in L_2$, $f\in L_1(0,T;L_2)$, $\mu_0=\nu_0\equiv 0$, $\nu_1\in L_2(B_T)$. Then  problem \eqref{3.1}, \eqref{1.2}--\eqref{1.4} has a unique weak solution $u\in X(Q_T)$.
\end{lemma}

\begin{proof}
We apply the contraction principle. For $t_0\in(0,T]$ define a mapping $\Lambda$ on $X(Q_{t_0})$ as follows: $u=\Lambda v\in X(Q_{t_0})$ is a weak solution to a linear problem
\begin{equation}\label{3.2}
u_t+bu_x+u_{xxx}+u_{xyy} =f-(g(v))_x -(\psi v)_x
\end{equation}
in $Q_{t_0}$ with initial and boundary conditions \eqref{1.2}--\eqref{1.4}.

Since 
\begin{gather*}
\|g(v)\|_{L_2(Q_{t_0})}\leq 
c\|v\|_{C([0,t_0];L_2)}<\infty, \\
\|\psi v\|_{L_2(Q_{t_0})}\leq c\|\psi\|_{L_2(0,t_0;L_\infty)}
\|v\|_{C([0,t_0];L_2)}<\infty,
\end{gather*}
Lemma~\ref{L2.11} provides that the mapping $\Lambda$ exists. Moreover, for functions $v,\widetilde{v}\in X(Q_{t_0})$
\begin{equation*}
\|g(v)-g(\widetilde{v})\|_{L_2(Q_{t_0})}\leq c\|v-\widetilde{v}\|_{L_2(Q_{t_0})}  \leq ct_0^{1/2}\|v-\widetilde{v}\|_{C([0,t_0];L_2)},
\end{equation*}
$$
\|\psi(v-\widetilde{v})\|_{L_2(Q_{t_0})} 
\leq c\|\psi\|_{L_2(0,t_0;L_\infty)} \|v-\widetilde{v}\|_{C([0,t_0];L_2)}.
$$
As a result, according to inequality \eqref{2.46}
\begin{equation*}
\|\Lambda v-\Lambda\widetilde{v}\|_{X(Q_{t_0})}\leq
c(T)\omega(t_0)\|v-\widetilde{v}\|_{X(Q_{t_0})},
\end{equation*}
where $\omega(t_0)\to 0$ as $t_0\to +0$ and $\omega$ depends on the properties of continuity of the primitive of the function $\|\psi(t,\cdot,\cdot)\|_{L_\infty}^2$ on $[0,T]$.
Since the constant in the right side of this inequality is uniform with respect to $u_0$ and $f$, one can construct the solution on the whole time segment $[0,T]$ by the standard argument.
\end{proof}

Now we pass to the results of existence in Theorem~\ref{T1.1}.

\begin{proof}[Proof of Existence Part of Theorem~\ref{T1.1}]
First of all we make zero the boundary data in \eqref{1.3} for the function $u$ itself. Let
\begin{equation}\label{3.3}
\psi(t,x,y)\equiv J(t,x,y;\mu_0)\eta(3/2-2x/R)  + J(-t,R-x,y;\nu_{0-})\eta(2x/R-1/2),
\end{equation}
where $\nu_{0-}(t,y)\equiv \nu_0(-t,y)$, the functions $\mu_0$ and $\nu_0$ are extended to the whole strip $B$ in the class $H^{s/3,s}(B)$, such that $\mu_0\equiv 0$ for $t<-1$, $\nu_0\equiv 0$ for $t>T+1$ and the function $J(t,x,y;\mu)$ is the aforementioned in the proof of Lemma~\ref{L2.10} solution to an initial-boundary value problem in $\Pi_T^+ =(0,T)\times \Sigma_+$ for equation \eqref{2.1} in the case $f\equiv 0$ with zero initial condition \eqref{1.2} for $(x,y)\in \Sigma_+$, boundary condition \eqref{1.4} for $(t,x)\in (0,T)\times\mathbb R_+$ and boundary condition \eqref{2.41}, introduced in \cite{F17}.
Then the results of \cite{F17} provide, that
\begin{equation}\label{3.4}
\left\{
\begin{aligned}
&\widetilde\psi \equiv\psi_t +b\psi_x+\psi_{xxx}+\psi_{xyy} \in C^\infty(\overline{Q}_T), \\
&\psi\in Y_0(Q_T)\cap L_2(0,T;W^1_\infty),\\
&\psi\big|_{x=0}=\mu_0, \quad \psi\big|_{x=R}=\nu_0, \quad \|\psi_x\big|_{x=R}\|_{L_2(B_T)}\leq c\|\nu_0\|_{H^{1/3,1}(B_T)}.
\end{aligned}
\right.
\end{equation}
Consider a function
\begin{equation}\label{3.5}
U(t,x,y) \equiv u(t,x,y) -\psi(t,x,y).
\end{equation}
Then $u\in X(Q_T)$ is a weak solution to problem \eqref{1.1}--\eqref{1.4} iff $U\in X(Q_T)$ is a weak solution to an initial-boundary value problem in $Q_T$ for an equation
\begin{equation}\label{3.6}
U_t+bU_x+U_{xxx}+U_{xyy}+UU_x+(\psi U)_x=F\equiv f-\widetilde\psi- \psi\psi_x,
\end{equation}
with initial and boundary conditions
\begin{equation}\label{3.7}
U\big|_{t=0} = U_0\equiv u_0 -\psi\big|_{t=0},\quad U\big|_{x=0}= U\big|_{x=R} =0,\quad
 U_x\big|_{x=R}=V_1\equiv \nu_1-\psi_x\big|_{x=R}
\end{equation}
and the same boundary conditions on $(0,T)\times (0,R)$ as \eqref{1.4}. Note also that the functions $U_0$, $F$, $V_1$ satisfy the same assumptions as the corresponding functions $u_0$, $f$, $\nu_1$ in the hypothesis of the theorem.
 
For $h\in (0,1]$ consider a set of initial-boundary value problems in $Q_T$ for an equation
\begin{equation}\label{3.8}
U_t+bU_x+U_{xxx}+U_{xyy}+\left(g_h(U)\right)_x+(\psi U)_x = F
\end{equation}
with boundary conditions \eqref{1.4} and \eqref{3.7}.
\begin{equation*}
g_h(u)\equiv\int_0^u\Bigl[\theta\eta(2-h|\theta|)+\frac{2\sgn\theta}{h}\eta(h|\theta|-1)\Bigr]\,d\theta.
\end{equation*}
Note that $g_h(u)= u^2/2$ if $|u|\leq 1/h$, $|g'_h(u)|\leq 2/h\ \forall u\in\mathbb R$ and $|g'_h(u)|\leq 2|u|$ uniformly with respect to $h$.

According to Lemma~\ref{L3.1}, there exists a unique solution to this problem $U_h\in X(Q_T)$.

Next, establish appropriate estimates for functions $U_h$ uniform with respect to~$h$ (we drop the index $h$ in intermediate steps for simplicity). First, note that $g'(U)U_x, \psi U_x, \psi_x U, F \in L_1(0,T;L_2)$ and so the hypothesis of Lemma~\ref{L2.11} is satisfied (for $f_1\equiv 0$). Write down the analogue of equality \eqref{2.47} for $\rho\equiv 1$, then:
\begin{equation}\label{3.9}
\iint U^2 \,dxdy \leq \iint U_0^2\,dxdy 
+\int_0^t \!\!\iint \bigl(2F- 2(g(U))_x-\psi_x U \bigr)U\,dxdyd\tau.
\end{equation}
Since
\begin{equation}\label{3.10}
(g(U))_xU=\partial_x\Bigl(\int_0^U g'(\theta)\theta \,d\theta\Bigr) 
\end{equation}
we derive that
\begin{equation}\label{3.11}
\iint (g(U))_x U \,dxdy = 0.
\end{equation}
Therefore, since $\psi_x\in L_2(0,T;L_\infty)$ uniformly with respect to $h$
\begin{equation}\label{3.12}
\|u_h\|_{C([0,T];L_2)} \leq c.
\end{equation}

Next, equalities \eqref{2.47} and \eqref{3.10} provide that for $\rho(x)\equiv (1+x)$
\begin{multline}\label{3.13}
\iint U^2\,dxdy + \int_0^t\!\! \iint (3U_x^2 +U_y^2)\,dxdyd\tau 
\leq (1+R)\iint U_0^2\,dxdy  \\ +b \int_0^t \!\!\iint U^2\, dxdyd\tau
+(1+R)\iint_{B_t} V_1^2\,dyd\tau
+2\int_0^t \!\!\iint FU\rho\,dxdyd\tau \\ 
+\int_0^t \!\!\iint (\psi-\psi_x\rho)U^2\,dxdyd\tau
+2\int_0^t \!\!\iint \Bigl(\int_0^U g'(\theta)\theta \,d\theta \Bigr)  \,dxdyd\tau.
\end{multline}
Note that 
\begin{equation}\label{3.14}
\Bigl|\int_0^U g'(\theta)\theta \,d\theta \Bigr| \leq c|U|^3.
\end{equation}
Applying interpolating inequality \eqref{1.15} (here the exact value of the constant is indifferent), we obtain that
\begin{equation}\label{3.15}
\iint |U|^3\,dxdy \leq 
c\iint U^2\,dxdy \Bigl(\iint \bigl(|DU|^2 +U^2\bigr)\,dxdy\Bigr)^{1/2}
\end{equation}
Since the norm of the functions $u_h$ in the space $L_2$ is already  estimated in \eqref{3.12}, it follows from \eqref{3.13}--\eqref{3.15} that uniformly with respect to $h$
\begin{equation}\label{3.16}
\|u_h\|_{X(Q_T)} \leq c.
\end{equation}

From equation \eqref{3.8} itself, estimate \eqref{3.14} and the well-known embedding $L_1\subset H^{-2}$, it follows that uniformly with respect to $h$
\begin{equation}\label{3.17}
\|u_{ht}\|_{L_1(0,T;H^{-3})}\leq c.
\end{equation}

Estimates \eqref{3.16}, \eqref{3.17} by the standard argument provide existence of a weak solution to problem \eqref{1.1}--\eqref{1.4} $u\in L_\infty(0,T;L_2)\cap L_2(0,T;\widetilde H^1)$, as a limit of functions $u_h$ when
$h\to +0$.

Finally, since by virtue of \eqref{1.20} (here the exact value of the constant is again indifferent)
\begin{multline}\label{3.18}
\iiint_{Q_T} U^4\,dxdydt \leq c\int_0^T \|U(t,\cdot,\cdot)\|^2_{H^1} \|U(t,\cdot,\cdot)\|^2_{L_2}\,dt \\ \leq
c\|U\|^2_{L_2(0,T;H^1)}\|U\|^2_{L_\infty(0,T;L_2)} <\infty
\end{multline}
and
\begin{equation}\label{3.19}
\iiint_{Q_T} \psi^2 U^2\,dxdydt \leq \|\psi\|^2_{L_2(0,T;L_\infty)} \|U\|^2_{L_\infty(0,T;L_2)} <\infty,
\end{equation}
it follows from Lemma~\ref{L2.11} (where $f_1\equiv U^2/2 +\psi U$), that after possible modification on a set of zero measure $U\in C([0,T];L_2)$.
\end{proof}

Result on uniqueness and continuous dependence of weak solutions succeeds from the following theorem.

\begin{theorem}\label{T3.1}
For any $T>0$ and $M>0$ there exist constant $c=c(T,M,b,R,L)$, such that for any two weak solutions $u(t,x,y)$ and $\widetilde u(t,x,y)$ to problem \eqref{1.1}--\eqref{1.4}, satisfying $\|u\|_{X(Q_T)}, \|\widetilde u\|_{X(Q_T)} \leq M$, with corresponding data $u_0, \widetilde u_0\in L_2$, $\mu_0, \widetilde\mu_0 ,\nu_0, \widetilde\nu_0 \in \widetilde H^{1/3,1}(B_T)$, $\nu_1,\widetilde\nu_1 \in L_2(B_T)$ $f, \widetilde f\in L_1(0,T;L_2)$ the following inequality holds:
\begin{multline}\label{3.20}
\|u -\widetilde u\|_{X(Q_T)} \leq c\bigl( \|u_0 - \widetilde u_0\|_{L_2} +
\|\mu_0-\widetilde\mu_0\|_{H^{1/3,1}(B_T)} \\+
\|\nu_0-\widetilde\nu_0\|_{H^{1/3,1}(B_T)} +\|\nu_1-\widetilde\nu_1\|_{L_2(B_T)}+
\|f-\widetilde f\|_{L_1(0,T;L_2)}\bigr).
\end{multline} 
\end{theorem}

\begin{proof}
Let the function $\psi$ is defined by formula \eqref{3.3}, the function $\widetilde\psi$ in a similar way for $\widetilde\mu_0$, $\widetilde\nu_0$ and $\Psi\equiv \psi-\widetilde\psi$. Then, in particular,
\begin{equation}\label{3.21}
\|\Psi\|_{X(Q_T)} \leq c\bigl(\|\mu_0-\widetilde\mu_0\|_{H^{1/3,1}(B_T)} +
\|\nu_0-\widetilde\nu_0\|_{H^{1/3,1}(B_T)}\bigr).
\end{equation}
Let $U_0\equiv u_0-\widetilde u_0 - \Psi\big|_{t=0}$, $F\equiv f-\widetilde f -(\Psi_t+b\Psi_x+\Psi_{xxx}+\Psi_{xyy})$, $V_1\equiv \nu_1-\widetilde\nu_1 -\Psi_x\big|_{x=R}$, then
\begin{equation}\label{3.22}
\|U_0\|_{L_2} \leq \|u_0-\widetilde u_0\|_{L_2} +
c\bigl(\|\mu_0-\widetilde\mu_0\|_{H^{1/3,1}(B_T)} + \|\nu_0-\widetilde\nu_0\|_{H^{1/3,1}(B_T)}\bigr), 
\end{equation}
\begin{multline}\label{3.23}
\|F\|_{L_1(0,T;L_2)} \leq \|f-\widetilde f\|_{L_1(0,T;L_2)} \\ +
c\bigl(\|\mu_0-\widetilde\mu_0\|_{H^{1/3,1}(B_T)} + \|\nu_0-\widetilde\nu_0\|_{H^{1/3,1}(B_T)}\bigr),
\end{multline}
\begin{equation}\label{3.24}
\|V_1\|_{L_2(B_T)} \leq \|\nu_1-\widetilde\nu_1\|_{L_2(B_T)} +
c\bigl(\|\mu_0-\widetilde\mu_0\|_{H^{1/3,1}(B_T)} + \|\nu_0-\widetilde\nu_0\|_{H^{1/3,1}(B_T)}\bigr).
\end{equation}
The function $U(t,x,y) \equiv u(t,x,y) -\widetilde u(t,x,y) -\Psi(t,x,y)$ is a weak solution to an initial-boundary value problem in $Q_T$ for an equation
$$
U_t+bU_x+U_{xxx}+U_{xyy} = F -(uu_x-\widetilde u \widetilde u_x)
$$
with initial and boundary conditions \eqref{1.4},
$$
U\big|_{t=0} =U_0,\qquad U\big|_{x=0}= U\big|_{x=R}=0, \qquad U_x\big|_{x=R}=V_1.
$$
Apply Lemma~\ref{L2.11} where $f_1\equiv -(u^2-\widetilde u^2)/2$. Note that similarly to \eqref{3.18} $f_1\in L_2(Q_T)$.
Therefore, we derive from \eqref{2.47} that for $t\in (0,T]$ and $\rho(x)\equiv (1+x)$
\begin{multline}\label{3.25}
\iint U^2\,dxdy +\int_0^t\!\! \iint (3U_x^2+U_y^2)\,dxdyd\tau \leq (1+R)\iint U_0^2\,dxdy  \\ 
+b\int_0^t\!\! \iint U^2\,dxdyd\tau + (1+R)\iint_{B_t} V_1^2\,dyd\tau + 2\int_0^t \!\!\iint FU\rho\, dxdyd\tau \\+  \int_0^t \!\!\iint (u^2-\widetilde u^2)(U\rho)_x\, dxdyd\tau
\end{multline}
Here $u^2-\widetilde u^2 = (u+\widetilde u)(U+\Psi)$ and by virtue of \eqref{1.20}
\begin{multline*}
\iint |u(U+\Psi)U_x|\,dxdy  \\ \leq c \Bigl(\iint u^4 \,dxdy \iint (U^4+\Psi^4) \,dxdy\Bigr)^{1/4} \Bigl(\iint U_x^2\,dxdy\Bigr)^{1/2} \\ \leq c_1 \|u\|^{1/2}_{H^1}\|u\|_{L_2}^{1/2}
\Bigl[\Bigl(\iint |DU|^2\,dxdy\Bigr)^{3/4} \Bigl(\iint U^2\,dxdy\Bigr)^{1/4}  +
\iint U^2\,dxdy \\+
\Bigl(\iint |DU|^2\,dxdy\Bigr)^{1/2}\|\Psi\|^{1/2}_{H^1}\|\Psi\|_{L_2}^{1/2}\Bigr]
\end{multline*}
and, therefore,
\begin{multline}\label{3.26}
\int_0^t\!\! \iint |u(U+\Psi)U_x|\,dxdyd\tau \leq
\varepsilon \int_0^t\!\! \iint |DU|^2\,dxdyd\tau + \int_0^t \|\Psi\|^2_{H^1}\,d\tau\\ + c(\varepsilon)
\int_0^t \gamma(\tau) \iint (U^2+\Psi^2)\,dxdyd\tau,
\end{multline}
where $\gamma(t)\equiv 1+\|u(t,\cdot,\cdot)\|^2_{H^1}\|u(t,\cdot,\cdot)\|^2_{L_2}\in L_1(0,T)$ and $\varepsilon>0$ can be chosen arbitrarily small. Then estimates \eqref{3.21}--\eqref{3.24}, \eqref{3.26} and inequality \eqref{3.25} provide the desired result.
\end{proof}

Finally, consider regular solutions.

\begin{lemma}\label{L3.3}
Let $g(u)\equiv u^2/2$, $\mu_0=\nu_0=\nu_1\equiv 0$, the functions $u_0$ and $f$ satisfy the hypothesis of Theorem~\ref{T1.2},  $\psi\in X^3(Q_T)$.  Then  problem \eqref{3.1}, \eqref{1.2}--\eqref{1.4} has a unique solution $u\in X^3(Q_T)$.
\end{lemma}

\begin{proof}
For $t_0\in(0,T]$, $v\in X^3(Q_{t_0})$ let $u=\Lambda v\in X^3(Q_{t_0})$ be a solution to a linear problem \eqref{3.2} (for $g(v)\equiv v^2/2$), \eqref{1.2}--\eqref{1.4}.

Apply Lemma~\ref{L2.15}. We have:
\begin{multline}\label{3.27}
\|vv_x+\psi v_x+\psi_x v\|_{C[0,t_0];L_2)} \leq
\|u_0u_{0x}+\psi\big|_{t=0}u_{0x}+\psi_x\big|_{t=0}u_0\|_{L_2} \\ +
\|(vv_x)_t+(\psi v)_{tx}\|_{L_1(0,t_0;L_2)}
\end{multline}
and with the use of \eqref{1.21} derive that 
\begin{gather}\label{3.28}
\|u_0u_{0x}\|_{L_2} \leq
c\|u_0\|_{L_{\infty}}\|u_{0x}\|_{L_2} \leq
c_1\|u_0\|_{\widetilde H^3}^2, \\
\label{3.29}
\|\psi\big|_{t=0}u_{0x}+\psi_x\big|_{t=0}u_0\|_{L_2} \leq
c\|\psi\big|_{t=0}\|_{H^1}\|u_0\|_{W^1_\infty} \leq c\|\psi\|_{X^3(Q_T)}
\|u_0\|_{\widetilde H^3};
\end{gather}
next,
\begin{equation}\label{3.30}
\|vv_{tx}\|_{L_1(0,t_0;L_2)} \leq 
\int_0^{t_0} \|v\|_{L_\infty}\|v_{tx}\|_{L_2}\,dt \\ \leq
ct_0^{1/2}\|v\|_{X^2(Q_{t_0})}
\|v\|_{X^3(Q_{t_0})},
\end{equation}
\begin{equation}\label{3.31}
\|v_xv_{t}\|_{L_1(0,t_0;L_2)} \leq
\int_0^{t_0} \|v_x\|_{L_4}\|v_t\|_{L_4}\,dt  \leq
ct_0^{1/2}\|v\|_{X^2(Q_{t_0})}
\|v\|_{X^3(Q_{t_0})}
\end{equation}
and similarly
\begin{equation}\label{3.32}
\|(\psi v)_{tx}\|_{L_1(0,t_0;L_2)} \leq 
ct_0^{1/2}\|\psi\|_{X^3(Q_{T})}
\|v\|_{X^3(Q_{t_0})}.
\end{equation}
Next,
\begin{equation}\label{3.33}
\|vv_{t}\|_{L_2(Q_{t_0})} \leq 
\Bigl(\int_0^{t_0} \|v\|^2_{L_\infty}\|v_{t}\|^2_{L_2}\,dt\Bigr)^{1/2}  \leq
ct_0^{1/2}\|v\|_{X^2(Q_{t_0})}
\|v\|_{X^3(Q_{t_0})},
\end{equation}
$(vv_x)_{yy} = vv_{xyy} +2v_yv_{xy}+v_xv_{yy}$, where similarly to \eqref{3.33}
\begin{equation}\label{3.34}
\|vv_{xyy}\|_{L_2(Q_{t_0})} \leq 
ct_0^{1/2}\|v\|_{X^2(Q_{t_0})}
\|v\|_{X^3(Q_{t_0})},
\end{equation}
\begin{equation}\label{3.35}
\|v_yv_{xy}\|_{L_2(Q_{t_0})} \leq 
\Bigl(\int_0^{t_0}  \|v_y\|^2_{L_4}\|v_{xy}\|^2_{L_4}\,dt \Bigr)^{1/2} \leq
ct_0^{1/2}\|v\|_{X^2(Q_{t_0})}
\|v\|_{X^3(Q_{t_0})}
\end{equation}
and similar estimate holds for $v_xv_{yy}$. Finally, similarly to \eqref{3.33}--\eqref{3.35}
\begin{equation}\label{3.36}
\|(\psi v)_t\|_{L_2(Q_{t_0})}+
\|(\psi v)_{xyy}\|_{L_2(Q_{t_0})} \\ \leq 
ct_0^{1/2}\|\psi\|_{X^3(Q_{T})}
\|v\|_{X^3(Q_{t_0})}.
\end{equation}
Moreover, the assumptions on the function $\psi$ ensure that the corresponding boundary conditions on the function $vv_x+(\psi v)_x$ are satisfied for $y=0$ and $y=L$. Therefore, the mapping $\Lambda$ exists and one can use estimate \eqref{2.60} to derive inequalities
\begin{equation}\label{3.37}
\|\Lambda v\|_{X^3(Q_{t_0})}\leq 
\widetilde c + ct_0^{1/2}\bigl(\|\psi\|_{X^3(Q_{T})}\|v\|_{X^3(Q_{t_0})}  
+\|v\|_{X^3(Q_{t_0})}^2\bigr),
\end{equation}
\begin{multline}\label{3.38}
\|\Lambda v - \Lambda\widetilde v\|_{X^3(Q_{t_0})}\leq ct_0^{1/2}
\Bigl(\|\psi\|_{X^3(Q_{T})}\|v-\widetilde v\|_{X^3(Q_{t_0})}  \\ 
+\bigl(\|v\|_{X^3(Q_{t_0})}+
\|\widetilde v\|_{X^3(Q_{t_0})}\bigr)
\|v-\widetilde v\|_{X^3(Q_{t_0})} \Bigr),
\end{multline}
where the constant $c$ depends on the parameters $T,b,R,L$ and the constant$\widetilde c$ also on the properties of functions $u_0$, $f$, $\psi$. Hence, existence of the unique solution to the considered problem in the space 
$X^3(Q_{t_0})$ on the time interval $[0,t_0]$, depending on 
$\|u_0\|_{\widetilde H^3}$, follows by the standard argument.

Now establish the following a priori estimate: if $u\in X^3(Q_{T'})$ is a solution to the considered problem for some $T'\in (0,T]$, then
\begin{equation}\label{3.39}
\|u\|_{X^3(Q_{T'})}  \leq c,
\end{equation}
where the constant $c$ depends on $T,b,R,L$ and the properties of the functions $u_0$, $f$, $\psi$ from the hypothesis of the present lemma.

It is already known, that (see \eqref{3.16})
\begin{equation}\label{3.40}
\|u\|_{X(Q_{T'})} \leq c.
\end{equation}

Apply Lemma~\ref{L2.12}, then by virtue of \eqref{2.51} for $\rho(x)\equiv 1+x$
\begin{multline}\label{3.41}
\iint u_y^2\,dxdy + \int_0^t\!\! \iint |Du_y|^2\,dxdyd\tau  \leq 
(1+R)\iint u_{0y}^2\,dxdy  \\
+ b\int_0^t\!\! \iint u_y^2 \, dxdyd\tau -
2\int_0^t\!\! \iint \bigl(f-uu_x-(\psi u)_x\bigr) u_{yy}\rho\,dxdyd\tau.
\end{multline}
Here for arbitrary $\varepsilon>0$
\begin{multline}\label{3.42}
2\int_0^t\!\!\iint uu_x u_{yy}\rho\,dxdyd\tau = \int_0^t\!\!\iint (u-u_x\rho)u_y^2\,dxdyd\tau \\ \leq
c\int_0^t \Bigl(\iint (u_x^2+u^2)\,dxdy \iint u_y^4\,dxdy\Bigr)^{1/2}\,d\tau \leq 
\varepsilon\int_0^t\!\! \iint \bigl(|Du_y|^2+u_y^2\bigr)\,dxdyd\tau  \\+
c(\varepsilon)\int_0^t \gamma(\tau)\iint u_y^2\,dxdy\,d\tau,
\end{multline}
where $\gamma\equiv\|u(t,\cdot,\cdot)\|^2_{\widetilde H^1}\in L_1(0,T')$,
\begin{multline}\label{3.43}
2\int_0^t\!\! \iint (\psi u)_x u_{yy}\rho\,dxdyd\tau  \\ \leq 
\sup\limits_{t\in [0,T]} \|\psi(t,\cdot,\cdot)\|_{W^1_\infty}
\int_0^t\!\!\Bigl(\iint u_{yy}^2\,dxdy \iint (u_x^2+u^2)\,dxdy\Bigr)^{1/2}\,d\tau\\
\leq  \varepsilon \int_0^t\!\!\iint u_{yy}^2\,dxdyd\tau +
c(\varepsilon)\|\psi\|^2_{X^3(Q_T)}\|u\|^2_{X(Q_{T'})}.
\end{multline}
Therefore, inequality \eqref{3.41} yields that
\begin{equation}\label{3.44}
\|u_y\|_{C([0,T'];L_2)} + \bigl\||Du_y|\bigr\|_{L_2(Q_{T'})} \leq c.
\end{equation}

Next, since the hypothesis of Lemma~\ref{L2.13} is fulfilled, write down the corresponding analogue of equality \eqref{2.47} for the function $u_t$ and $\rho(x)\equiv 1+x$:
\begin{multline}\label{3.45}
\iint u_t^2\,dxdy + \int_0^t \!\! \iint (3u_{tx}^2+u_y^2)\,dxdyd\tau \\ \leq
(1+R)\iint \bigl(f-bu_x-u_{xxx}-u_{xyy}-uu_x-(\psi u)_x\bigr)^2\big|_{t=0}\,dxdy +b\int_0^t \!\! \iint u_t^2\,dxdyd\tau \\ + 2\int_0^t \!\! \iint f_t u_t\rho\,dxdyd\tau  
+2\int_0^t \!\! \iint \bigl(uu_t+(\psi u)_t\bigr)(u_t\rho)_x\,dxdyd\tau.
\end{multline}
Here similarly to \eqref{3.42}, \eqref{3.43} for arbitrary $\varepsilon>0$
\begin{multline*}
2\int_0^t\!\! \iint uu_t(u_t\rho)_x\,dxdyd\tau = \int_0^t\!\! \iint (u-u_x\rho)u_t^2\,dxdyd\tau  \\\leq 
\varepsilon\int_0^t\!\! \iint \bigl(|Du_t|^2+u_t^2\bigr)\,dxdyd\tau  +
c(\varepsilon)\int_0^t \gamma(\tau)\iint u^2_t\,dxdy\,d\tau,
\end{multline*}
where $\gamma\equiv\|u(t,\cdot,\cdot)\|^2_{\widetilde H^1}\in L_1(0,T')$,
\begin{multline*}
2\int_0^t\!\! \iint \psi_t u (u_t\rho)_x\,dxdyd\tau  \\ \leq 
c\int_0^t\!\!\Bigl(\iint (u_{tx}^2+u_t^2)\,dxdy\Bigr)^{1/2} 
\Bigl(\iint \psi_t^4\,dxdy \iint u^4\,dxdy\Bigr)^{1/4}\,d\tau\\ \leq
\varepsilon \int_0^t\!\!\iint (u_{tx}^2+u_t^2)\,dxdyd\tau
+c(\varepsilon)\|\psi_t\|^2_{X(Q_T)}\|u\|^2_{X(Q_{T'})},
\end{multline*}
and 
\begin{multline*}
2\int_0^t\!\! \iint \psi u_t (u_t\rho)_x\,dxdyd\tau = \int_0^t\!\! \iint(\psi-\psi_x\rho) u_t^2\,dxdyd\tau \\ \leq
c\|\psi\|_{X^3(Q_T)} \int_0^t\!\!\iint u_t^2\,dxdyd\tau.
\end{multline*}
Consequently, it follows from \eqref{3.45}, that
\begin{equation}\label{3.46}
\|u_t\|_{X(Q_{T'})} \leq c.
\end{equation}

Now apply Lemma~\ref{L2.14}, then inequality \eqref{2.52} and estimates \eqref{3.40}, \eqref{3.44} and \eqref{3.46} yield that for any $t\leq T'$ and$\rho(x)\equiv 1+x$
\begin{multline}\label{3.47}
\|u\|^2_{X^2(Q_t)} \leq c+
c\|uu_x\|^2_{C([0,t];L_2)}+c\|(\psi u)_x\|^2_{C([0,t];L_2)} \\ +
c \sup\limits_{\tau\in (0,t]}\Bigl|\int_0^\tau \!\! \iint \bigl(uu_x+(\psi u)_x\bigr)_{yy}u_{yy}\rho\,dxdyds\Bigr|.
\end{multline}
Uniformly with respect to $t\in [0,T']$ for arbitrary $\varepsilon>0$
\begin{gather*}
\|uu_x\|^2_{L_2} \leq c \|u\|^2_{\widetilde H^1}\|u_x\|^2_{L_4} \leq
\varepsilon \bigl\| |Du_x| \bigr\|^2_{L_2} +c(\varepsilon)\bigl(\|u_t\|^6_{X(Q_{T'})}+\|u\|^6_{X(Q_{T'})}+1\bigr),\\
\|(\psi u)_x\|^2_{L_2} \leq \|\psi\|^2_{W^1_\infty} \|u\|^2_{\widetilde H^1} \leq c\bigl(\|u_t\|^2_{X(Q_{T'})}+
\|u\|^2_{X(Q_{T'})}\bigr);
\end{gather*}
then,
\begin{equation*}
\iint (uu_x)_{yy}u_{yy}\rho\,dxdy = \frac12\iint (u_x\rho-u)u^2_{yy}\,dxdy +
2\iint u_yu_{xy}u_{yy}\rho\,dxdy,
\end{equation*}
where 
\begin{multline*}
\int_0^t\!\!\iint |u_yu_{xy}u_{yy}|\,dxdyd\tau  \leq  \sup\limits_{t\in [0,T']}\iint u_y^2\,dxdy 
\int_0^t\!\!\Bigl(\iint (u_{xy}^4+u_{yy}^4)\,dxdy\Bigr)^{1/2}d\tau \\ \leq
\varepsilon\int_0^t\!\! \iint |D^3 u|^2\,dxdyd\tau +
c(\varepsilon)\bigl(\|u_t\|^2_{X(Q_{T'})}+
\|u\|^2_{X(Q_{T'})}\bigr)\int_0^t \iint |D^2 u|^2\,dxdyd\tau,
\end{multline*}
\begin{multline*}
\int_0^t\!\!\iint |u-u_x\rho|u_{yy}^2\,dxdyd\tau  \leq c\sup\limits_{t\in [0,T']}\|u\|^2_{H^1} 
\int_0^t \Bigl(\iint u_{yy}^4\,dxdy\Bigr)^{1/2}\,d\tau \\ \leq 
\varepsilon\int_0^t\!\! \iint |Du_{yy}|^2\,dxdyd\tau  +
c(\varepsilon)\bigl(\|u_t\|^2_{X(Q_{T'})}+
\|u\|^2_{X(Q_{T'})}\bigr)\int_0^t \!\!\iint u_{yy}^2\,dxdy\,d\tau;
\end{multline*}
finally, $(\psi u)_{xyy} = \psi_{xyy}u +2\psi_{xy}u_y +\psi_{yy}u_x +\psi_xu_{yy} +2\psi_yu_{xy} +\psi u_{xyy}$, where
\begin{multline*}
\int_0^t\!\! \iint |\psi_{xyy}u u_{yy}|\,dxdtd\tau  \\ \leq \sup\limits_{t\in [0,T]}\|\psi_{xyy}\|_{L_2} 
\int_0^t \Bigl(\iint u^4\,dxdy \iint u_{yy}^4\,dxdy\Bigr)^{1/4}\,d\tau \\ \leq
\varepsilon\int_0^t\!\!\iint \Bigl(|D^3 u|^2+|D^2 u|^2\bigr)\,dxdyd\tau +
c(\varepsilon)\|\psi\|^2_{X^3(Q_T)}\|u\|^2_{X(Q_{T'})},
\end{multline*}
\begin{multline*}
\int_0^t\!\! \iint |\psi_{xy}u_y u_{yy}|\,dxdtd\tau  \\ \leq
\int_0^t \!\! \Bigl(\iint \psi_{xy}^4\,dxdy\Bigr)^{1/4} \Bigl(\iint u_y^2\,dxdy\Bigr)^{1/2} \Bigl(\iint u_{yy}^4\,dxdy\Bigr)^{1/4}\,d\tau \\ \leq
\varepsilon\int_0^t\!\! \iint \Bigl(|D^3 u|^2+|D^2 u|^2\bigr)\,dxdyd\tau +
c(\varepsilon)\|\psi\|^2_{X^3(Q_T)}\|u\|^2_{X(Q_{T'})}
\end{multline*}
and similar estimate holds for the integral of $\psi_{yy}u_xu_{yy}$. The rest integrals are estimated in an obvious way. 
As a result, it follows from \eqref{3.47} that
\begin{equation}\label{3.48}
\|u\|_{X^2(Q_{T'})} \leq c.
\end{equation}

Finally, apply Lemma~\ref{L2.15} on the basis of the already obtained estimates \eqref{3.46}, \eqref{3.48}, then inequality \eqref{2.60} and estimates \eqref{3.27}--\eqref{3.36} applied to $v\equiv u$  provide similarly to \eqref{3.37} that for any $t_0\in (0,T']$
\begin{equation*}
\|u\|_{X^3(Q_{t_0})}\leq 
\widetilde c + ct_0^{1/2}\bigl(\|\psi\|_{X^3(Q_{T})}+\|u\|_{X^2(Q_{T'})}\bigr)
\|u\|_{X^3(Q_{t_0})},
\end{equation*}
whence \eqref{3.39} follows.
\end{proof}

\begin{proof}[Proof of Theorem~\ref{T1.2}]
Let $\psi\in Y^3(Q_T)\subset X^3(Q_T)$ be the solution to problem \eqref{2.1}, \eqref{1.2}--\eqref{1.4} for $f\equiv 0$ (see Lemma~\ref{L2.12}). Introduce the function $U$ by formula \eqref{3.5} and consider problem \eqref{3.6}, \eqref{3.7}, \eqref{1.4} (here $\widetilde\psi\equiv 0$, $V_1\equiv 0$). Then the functions $\psi$, $F\sim f$ and $U_0\sim u_0$ satisfy the hypothesis of Lemma~\ref{L3.3} and the result is immediate.
\end{proof}

\section{Large-time decay of small solutions}\label{S4}

\begin{proof}[Proof of Theorem~\ref{T1.3}]
Consider the solution to problem \eqref{1.1}--\eqref{1.4} $u\in X(Q_T)\ \forall T$. Note that $u^2\in L_2(Q_T)$ (see, for example, \eqref{3.18}). Apply Lemma~\ref{L2.11}, then equality \eqref{2.47} for $f_1\equiv u^2/2$, 
$\rho\equiv 1$  and equality \eqref{3.10} for $g(u)\equiv u^2/2$ yield similarly to \eqref{3.12}, that
\begin{equation}\label{4.1}
\|u(t,\cdot,\cdot)\|^2_{L_2} \leq \|u_0\|^2_{L_2}+\|\nu_1\|^2_{L_2(B_+)}\leq \epsilon_0^2\quad \forall\ t\geq 0.
\end{equation}
Next, it follows from equality \eqref{2.47} for $\rho\equiv 1+x$, that
\begin{multline}\label{4.2}
\iint u^2\rho\,dxdy + \iint_{B_t} \mu_1^2\,dyd\tau +
\int_0^t \!\!\iint (3u_x^2+u_y^2-bu^2)\,dxdyd\tau \\
= \iint u_0^2\rho\,dxdy + (1+R)\iint_{B_t} \nu_1^2\,dyd\tau+ \frac{2}3 \int_0^t\!\!\iint u^3\,dxdyd\tau.
\end{multline} 
Since $u^3\in L_1(Q_T)$ equality \eqref{4.2} provides the following inequality in a differential form:
for a.e. $t>0$
\begin{equation}\label{4.3}
\frac{d}{dt}\iint u^2\rho\,dxdy + \iint (3u_x^2+u_y^2-bu^2)\,dxdy 
\leq (1+R)\int_0^L \nu_1^2\,dy+\frac{2}3 \iint u^3\,dxdy.
\end{equation} 
Next, we show that inequality \eqref{4.3} implies the following one:
\begin{multline}\label{4.4}
\frac{d}{dt}\iint u^2\rho\,dxdy +\frac{\varkappa}{1+R}\iint u^2\rho\,dxdy \\
+\delta\iint \Bigl[1-\frac1{\varepsilon_0}\|u(t,\cdot,\cdot)\|_{L_2}\Bigr](3u_x^2+u_y^2)\,dxdy \leq 
(1+R)\int_0^L \nu_1^2\,dy.
\end{multline}
where $\delta$, $\varkappa$ and $\epsilon_0$ are from the hypothesis of the theorem. First of all note, that in all cases inequality \eqref{1.17} implies, that
\begin{equation}\label{4.5}
\iint u_x^2\,dxdy \geq \frac {\pi^2}{R^2} \iint u^2\,dxdy.
\end{equation}
Further consider different cases separately. 

In the cases b) and d) it follows from inequality \eqref{4.5}, that 
\begin{equation}\label{4.6}
(1-\delta) \iint (3u_x^2+u_y^2)\,dxdy -b\iint u^2\,dxdy \geq \frac{\varkappa}{1+R}\iint u^2\rho\,dxdy.
\end{equation}
Moreover, by virtue of \eqref{1.15} and \eqref{4.5}
\begin{multline}\label{4.7}
\frac 23 \iint u^3\,dxdy \leq \frac {4R}{3\pi} \Bigl(\iint u_x^2\,dxdy\Bigr)^{3/4} \Bigl(\iint u_y^2\,dxdy\Bigr)^{1/4}\Bigl( \iint u^2\,dxdy\Bigr)^{1/2} \\ +
\frac {4R^{3/2}}{3L^{1/2}\pi^{3/2}} \iint u_x^2\,dxdy\Bigl( \iint u^2\,dxdy\Bigr)^{1/2} \leq
\frac{\delta}{\epsilon_0}\|u(t,\cdot,\cdot)\|_{L_2}\iint (3u_x^2+u_y^2)\,dxdy,
\end{multline} 
and \eqref{4.4} follows.

In the case a) we also use an inequality
\begin{equation}\label{4.8}
\iint u_y^2\, dxdy \geq \frac {\pi^2}{L^2} \iint u^2\,dxdy
\end{equation}
and, therefore, obtain \eqref{4.6} with the corresponding $\varkappa$.
Then we can alternatively derive, that either similarly to \eqref{4.7}
\begin{multline*}
\frac 23 \iint u^3\,dxdy \leq \frac {4R}{3\pi} \Bigl(\iint u_x^2\,dxdy\Bigr)^{3/4} \Bigl(\iint u_y^2\,dxdy\Bigr)^{1/4}\Bigl( \iint u^2\,dxdy\Bigr)^{1/2} \\ \leq
\frac{4R}{3^{7/4}\pi}\|u(t,\cdot,\cdot)\|_{L_2}\iint (3u_x^2+u_y^2)\,dxdy,
\end{multline*}
or
\begin{multline}\label{4.9}
\frac 23 \iint u^3\,dxdy \leq \frac {4L}{3\pi} \Bigl(\iint u_x^2\,dxdy\Bigr)^{1/4} \Bigl(\iint u_y^2\,dxdy\Bigr)^{3/4}\Bigl( \iint u^2\,dxdy\Bigr)^{1/2} \\ \leq
\frac{4L}{3^{5/4}\pi}\|u(t,\cdot,\cdot)\|_{L_2}\iint (3u_x^2+u_y^2)\,dxdy,
\end{multline}
whence \eqref{4.4} follows.

In the case c) inequality \eqref{4.8} must be substituted by the following one:
$$
\iint u_y^2\, dxdy \geq \frac {\pi^2}{4L^2} \iint u^2\,dxdy.
$$
Similar modification must be done in \eqref{4.9} and \eqref{4.4} in this case also follows.

Inequalities \eqref{4.1} and \eqref{4.4} imply, that
$$
\frac{d}{dt}\iint u^2\rho\,dxdy +\frac{\varkappa}{1+R}\iint u^2\rho\,dxdy \leq (1+R)\int_0^L \nu_1^2\,dy,
$$
whence \eqref{1.11} easily succeeds.
\end{proof}

\section{Boundary controllability}\label{S5}

First establish the result on boundary controllability for the linear equation. 

\begin{theorem}\label{T5.1}
Let condition \eqref{1.13} be satisfied for any natural $l$, such that $\lambda_l <b$.
Let $T>0$, $f\equiv 0$, $\mu_0=\nu_0\equiv 0$. Then for any $u_0, u_T \in L_2$ there exists a function $\nu_1\in L_2(B_T)$, such that there exists a unique solution $u\in Y_0(Q_T)$ to problem \eqref{2.1}, \eqref{1.2}--\eqref{1.4}, satisfying \eqref{1.12}.
\end{theorem}

\begin{proof}
Assume first that $u_0\equiv 0$. In the case $\nu_1\in L_2(B_T)$, $u_0\equiv 0$, $\mu_0=\nu_0\equiv 0$, $f\equiv 0$ denote the solution $u\in Y_0(Q_T)$ to problem \eqref{2.1}, \eqref{1.2}--\eqref{1.4} by $P_1\nu_1$. Then Lemma~\ref{L2.10} provides, that $P_1$ is the linear bounded operator from $L_2(B_T)$ to $Y_0(Q_T)$. 

Let $P_{1T}\nu_1 \equiv P_1\nu_1\big|_{t=T}$, then $P_{1T}$ is the linear bounded operator from $L_2(B_T)$ to $L_2$.

Consider also the backward problem in $Q_T$
\begin{gather}\label{5.1}
\phi_t+b\phi_x+\phi_{xxx}+\phi_{xyy}=0,\\
\label{5.2}
\phi\big|_{t=T}=\phi_0(x,y), \quad \phi\big|_{x=0}=\phi_x\big|_{x=0}=\phi\big|_{x=R}=0
\end{gather}
with corresponding boundary conditions of \eqref{1.4} type, which after change of variables $(t,x,y)\to (T-t,R-x,y)$ transforms to the corresponding problem of \eqref{2.1}, \eqref{1.2}--\eqref{1.4} type. In particular, if we denote $\phi=\widetilde P\phi_0$, then $\widetilde P$ is the linear bounded operator from $L_2$ to $Y_0(Q_T)$. Moreover estimates 
\eqref{2.66}, \eqref{2.67} yield, that for $\Lambda\phi_0 \equiv \partial_x(\widetilde P\phi_0)\big|_{x=R}$
\begin{equation}\label{5.3}
\|\Lambda\phi_0\|_{L_2(B_T)} \leq \|\phi_0\|_{L_2}
\leq c \|\Lambda\phi_0\|_{L_2(B_T)}.
\end{equation}
In the smooth case multiplying equation \eqref{5.1} by $P_1\nu_1$ and integrating over $Q_T$ one can easily derive an equality
\begin{equation}\label{5.4}
\iint P_{1T}\nu_1\cdot\phi_0\,dxdy = \iint_{B_T} \nu_1\cdot\Lambda\phi_0\,dydt.
\end{equation}
By continuity this equality can be extended to the case $\nu_1\in L_2(B_T)$, $\phi_0\in L_2$.
Let $A \equiv P_{1T}\circ\Lambda$, then according to \eqref{5.3} and the aforementioned properties of the operator $P_{1T}$ the operator $A$ is bounded in $L_2$. Moreover, \eqref{5.3} and \eqref{5.4} provide, that
$$
(A\phi_0,\phi_0) = \iint (P_{1T}\circ\Lambda)\phi_0\cdot \phi_0\,dxdy
= \iint_{B_T} (\Lambda\phi_0)^2\,dydt \geq \frac 1{c^2} \|\phi_0\|^2_{L_2}.
$$
Application of Lax--Milgram theorem implies, that $A$ is invertible and $A^{-1}=\Lambda^{-1}\circ P_{1T}^{-1}$ is bounded in $L_2$. Let
\begin{equation}\label{5.5}
\Gamma \equiv \Lambda\circ A^{-1} = P_{1T}^{-1}
\end{equation}
(linear bounded operator from $L_2$ to $L_2(B_T)$), then $\nu_1\equiv \Gamma u_T$ and $u\equiv P_1\nu_1$ provide the desired solution in the case $u_0\equiv 0$.

In the general case the solution is given by the formula
\begin{equation}\label{5.6}
\nu_1\equiv \Gamma(u_T-Pu_0\big|_{t=T}),\quad u\equiv Pu_0 +P_1\nu_1
\end{equation}
(remind that $Pu_0$ is the solution to problem \eqref{2.1}, \eqref{1.2}--\eqref{1.4} for $\mu_0=\nu_0=\nu_1\equiv 0$, $f\equiv 0$).
\end{proof}

Now we can prove Theorem~\ref{T1.4}.

\begin{proof}[Proof of Theorem~\ref{T1.4}]
Consider first linear problem \eqref{2.1}, \eqref{1.2}--\eqref{1.4}. Let $u_0\equiv 0$, $\mu_0=\nu_0=\nu_1\equiv 0$, $f\equiv f_{1x}$, $f_1\in L_2(Q_T)$. Let $P_2f_1\in X(Q_T)$ be the solution to this problem, existing by virtue of Lemma~\ref{L2.11}. In particular, estimate \eqref{2.46} yields, that $P_2$ is the linear bounded operator from $L_2(Q_T)$ to $X(Q_T)$.

Obviously, a solution $\nu_1\in L_2(B_T)$, $u\in X(Q_T)$ to the controllability problem
\begin{gather*}
u_t+bu_x+u_{xxx}+u_{xyy}=f_{1x},\quad f_1\in L_2(Q_T),\\
u\big|_{t=0}=u_0\in L_2,\quad u\big|_{t=T}=u_T\in L_2, \quad 
u\big|_{x=0}=u\big|_{x=R}=0,\quad u_x\big|_{x=R}=\nu_1
\end{gather*}
is given by the formula
\begin{equation}\label{5.7}
\nu_1\equiv \Gamma\bigl(u_T-Pu_0\big|_{t=T}-P_2f_1\big|_{t=T}\bigr), \quad
u\equiv Pu_0 +P_1\nu_1 +P_2f_1.
\end{equation}

The solution to the original problem is constructed as a fixed point of the map
\begin{equation}\label{5.8}
u=\Theta v\equiv Pu_0 + (P_1\circ\Gamma)\bigl(u_T-Pu_0\big|_{t=T}+P_2(v^2/2)\big|_{t=T}\bigr) -P_2(v^2/2),
\end{equation}
defined on $X(Q_T)$. Similarly to \eqref{3.18}
\begin{gather*}
\|v^2\|_{L_2(Q_T)} \leq c\|v\|^2_{X(Q_T)}, \\
\|v^2-\widetilde v^2\|_{L_2(Q_T)} \leq c\bigl(\|v\|_{X(Q_T)}+\|\widetilde v\|_{X(Q_T)}\bigr)
\|v-\widetilde v\|_{X(Q_T)}.
\end{gather*}
Therefore,
\begin{gather*}
\|\Theta v\|_{X(Q_T)} \leq c\bigl(\|u_0\|_{L_2} + \|u_T\|_{L_2} + \|v\|^2_{X(Q_T)}\bigr),\\
\|\Theta v -\Theta \widetilde v\|_{X(Q_T)} \leq 
c\bigl(\|v\|_{X(Q_T)}+\|\widetilde v\|_{X(Q_T)}\bigr)
\|v-\widetilde v\|_{X(Q_T)}
\end{gather*}
and the standard contraction argument provides the desired result.
\end{proof}


\begin{thebibliography}{99}

\bibitem{BF13} E.~S.~Baykova and A.~V.~Faminskii, {\it On initial-boundary-value problems in a strip for the generalized two-dimensional Zakharov--Kuznetsov equation}, Adv. Differential Equ. {\bf 18} (2013), 663-686.

\bibitem{BL} H.~A.~Biagioni and F.~Linares, \textit{Well-posedness for the modified Zakharov--Kuznetsov equation}, Progr. Nonlinear Differential Equ. Appl. \textbf{54} (2003), 181--189.

\bibitem{BJM} E.~Bustamante, J.~Jimenez and J.~Mejia, \textit{The Zakharov--Kuznetsov equation in weighted Sobolev spaces}, J. Math. Anal. Appl. \textbf{433} (2016), 149--175.

\bibitem{DL} G.~G.~Doronin and N.~A.~Larkin, {\it Stabilization of regular solutions for the Zakharov--Kuznetsov equation posed on bounded rectangles and on a strip}, Proc. Edinburgh Math. Soc. {\bf 58} (2015), 661--682.

\bibitem{F89} A.~V.~Faminskii, {\it The Cauchy problem for quasilinear equations of odd order}, Mat. Sb. {\bf 180} (1989), 1183--1210. English transl. in Math. USSR-Sb. {\bf 68} (1991), 31--59.

\bibitem{F95} A.~V.~Faminski, {\it The Cauchy problem for the Zakharov--Kuznetsov equation}, Differ. Uravn. {\bf 31} (1995), 1070--1081. English transl. in Differential Equ. {\bf 31} (1995), 1002--1012.

\bibitem{F02} A .~V.~Faminskii, {\it On the mixed problem for quasilinear equations of the third order}, J. Math. Sci. {\bf 110} (2002), 2476--2507.

\bibitem{F07-1} A.~V.~Faminskii, {\it On the nonlocal well-posedness of a mixed problem for the Zakharov--Kuznetsov equation}, J. Math. Sci. {\bf 147} (2007), 6524--6537. 

\bibitem{F07-2} A.~V.~Faminskii, {\it Global well-posedness of two initial-boundary-value problems for the Korteweg--de~Vries equation}, Differential Integral Equ. {\bf 20} (2007), 601--642.

\bibitem{F08} A.~V.~Faminskii, {\it Well-posed initial-boundary value problems for the Zakharov--Kuznetsov equation}, Electronic J. Differential Equ. No.~127 (2008), 1--23.

\bibitem{FB08} A.~V.~Faminskii and I.~Yu.~Bashlykova, {\it Weak solutions to one initial-boundary value problem with three boundary conditions for quasilinear equations of the third order}, Ukrainian Math. Bull. {\bf 5} (2008), 83--98.

\bibitem{F12} A.~V.~Faminskii, {\it Weak solutions to initial-boundary-value problems for quasilinear evolution equations of an odd order}, Adv. Differential Equ. {\bf 17} (2012), 421--470.

\bibitem{F15-1} A.~V.~Faminskii, \textit{An initial-boundary value problem in a strip for two-dimensional Zakharov--Kuznetsov--Burgers equation}, Nonlinear Analysis \textbf{116} (2015), 132--144.

\bibitem{F15-2} A.~V.~Faminskii, {\it An initial-boundary value problem in a strip for two-dimensional equations of Zakharov--Kuznetsov type}, Contemp. Math. {\bf 653} (2015), 137--162. 

\bibitem{F17} A.~V.~Faminski, {\it Initial-boundary value problems in a half-strip for two-dimensional Zakharov--Kuznetsov equation}, arXiv: 1703.05660v1 [math.AP] 16 Mar 2017. 

\bibitem{FLP} L.~G.~Farah, F.~Linares and A.~Pastor, {\it A note on the 2D generalized Zakharov--Kuznetsov equation: local, global and scattering results}, J. Differential Equ. {\bf 253} (2012), 2558--2571.

\bibitem{FP} G.~Fonseca and M.~Panch\'on, \textit{Well-posedness for the two dimensional generalized Zakharov--Kuznetsov equation in anisotropic weighted Sobolev spaces}, J. Math. Anal. Appl. \textbf{443} (2016), 566--584.

\bibitem{G} A.~Gr\"unrock, \textit{On the generalized Zakharov--Kuznetsov equation at critical regularity}, arXiv: 1509.09146v1 [math.AP] 30 Sep 2015. 

\bibitem{GH} A.~Gr\"unrock and S.~Herr, \textit{The Fourier restriction norm method for the Zakharov--Kuznetsov equation}, Discrete Cont. Dyn. Syst. (A) \textbf{34} (2014), 2061--2068.

\bibitem{H-K} D.~Han-Kwan, \textit{From Vlasov--Poisson to Korteweg--de~Vries and Zakharov--Kuznetsov}, Comm. Math. Phys. \textbf{324} (2013), 961--993.

\bibitem{LSU} O.~A.~Ladyzhenskaya, V.~A.~Solonnikov and N.~N.~Uraltseva, {\it Linear and quasilinear equations of parabolic type}, Trans. of Math. Monogr. {\bf 23}, American Math. Soc.,  Providence, R.I., 1968.

\bibitem{LLS} D.~Lannes, F.~Linares and J.-C.~Saut, {\it The Cauchy problem for the Euler-Poisson system and derivation of the Zakharov--Kuznetsov equation}, Progress Nonlinear Differential Equ. Appl. {\bf 84} (2013), 183--215.

\bibitem{L13} N.~A.~Larkin, {\it Exponential decay of the $H^1$-norm for the 2D Zakharov--Kuznetsov equation}, J. Math. Anal. Appl. {\bf 405} (2013), 326--335.

\bibitem{L15} N.~A.~Larkin, \textit{The 2D Zakharov--Kuznetsov--Burgers equation with variable dissipation on a strip}, Electronic J. Differential Equ. (2015), no.~60, 1--20.

\bibitem{L16} N.~A.~Larkin, \textit{The 2D Zakharov--Kuznetsov--Burgers equation on a strip}, Bol. Soc. Parana Mat. (3) \textbf{34} (2016), 151--172.

\bibitem{LT} N.~A.~Larkin and E.~Tronco, {\it Regular solutions of the 2D Zakharov--Kuznetsov equation on a half-strip}, J. Differential Equ. {\bf 254} (2013), 81--101.

\bibitem{LP09} F.~Linares and A.~Pastor, {\it Well-posedness for the two-dimensional modified Zakharov--Kuznetsov equation}, SIAM J. Math. Anal. {\bf 41} (2009), 1323--1339.

\bibitem{LP11} F.~Linares and A.~Pastor, {\it Well-posedness for the 2D modified Zakharov--Kuznetsov equation}, J. Funct. Anal., {\bf 260} (2011), 1060--1085.

\bibitem{LPS} F.~Linares, A.~Pastor and J.-C.~Saut, {\it Well-posedness for the Zakharov--Kuznetsov equation in a cylinder and on the background of a KdV soliton}, Comm. Partial Differential Equ., {\bf 35} (2010), 1674--1689.

\bibitem{LM} J.-L.~ Lions and E.~ Magenes, \textit{Probl\`emes aux limites non homog\`enes
et applications}, Dunod, Paris, 1968.

\bibitem{MP} L.~Molinet and D.~Pilod, \textit{Bilinear Strichartz estimates for the Zakharov--Kuznetsov equation and applications}, Ann. Inst. H.~Poincare (C) Analyse Non Lin\'eaire \textbf{32} (2015), 347--371.

\bibitem{RV} F.~Ribaud and S.~Vento, {\it A note on the Cauchy problem for the 2D generalized Zakharov--Kuznetsov equation}, C. R. Acad. Sci. Paris {\bf 350} (2012), 499--503.

\bibitem{R} L.~Rosier, {\it Exact boundary controllability for the Korteweg--de~Vries equation on a bounded domain}, ESAIM: Control, Optimization Calculus Variations {\bf 2} (1997), 33--55.

\bibitem{S} J.-C.~Saut, {\it Sur quelques generalizations de l'equation de Korteweg--de~Vries}, J. Math. Pures Appl., {\bf 58} (1979), 21--61.

\bibitem{ST} J.-C.~Saut and R.~Temam, {\it An initial boundary value problem for the Zakharov--Kuznetsov equation}, Adv. Differential Equ., {\bf 15} (2010), 1001--1031.

\bibitem{STW} J.-C.~Saut, R.~Temam and C.~Wang, {\it An initial and boundary-value problem for the Zakharov--Kuznetsov equation in a bounded domain}, J. Math. Phys. {\bf 53} (2012), 115612.

\bibitem{ZK} V.E.~Zakharov and E.A.~Kuznetsov, {\it On three-dimensional solutions}, Zhurnal Eksp. Teoret. Fiz., {\bf 66} (1974), 594--597. English transl. in Soviet Phys. JETP, {\bf 39} (1974), 285--288.

\end{thebibliography}
\end{document}